\documentclass[11pt, a4paper]{amsart}
\usepackage{amsmath}
\usepackage{amsfonts}
\usepackage{amssymb}
\usepackage{amsthm}
\usepackage{mathrsfs}
\usepackage{enumerate}
\usepackage{hyperref}
\usepackage{cite}

\usepackage[T1]{fontenc}
\usepackage[utf8]{inputenc}

\DeclareMathOperator{\Leb}{\operatorname{Leb}}
\DeclareMathOperator{\CE}{\mathcal{C}\mathcal{E}}
\DeclareMathOperator{\NR}{\mathcal{N}\mathcal{R}}
\DeclareMathOperator{\PR}{\mathcal{P}\mathcal{R}}
\DeclareMathOperator{\R}{\mathcal{R}}

\theoremstyle{plain}
\newtheorem{Th}{Theorem}[section]
\newtheorem{Lemma}[Th]{Lemma}
\newtheorem{Cor}[Th]{Corollary}
\newtheorem{Prop}[Th]{Proposition}

\newtheorem{Thx}{Theorem}

\theoremstyle{definition}
\newtheorem{Def}[Th]{Definition}

\newtheorem{Rem}[Th]{Remark}
\newtheorem{?}[Th]{Problem}

\usepackage{microtype}

\newtheorem*{acknowledgement}{Acknowledgement}

\title{Critical recurrence in the real quadratic family}
\author{Mats Bylund}

\address{Centre for Mathematical Sciences, Lund University, Box 118, 221 00 Lund, Sweden}
\email{mats.bylund@math.lth.se}


\begin{document}

\begin{abstract}
We study recurrence in the real quadratic family and give a sufficient condition on the recurrence rate $(\delta_n)$ of the critical orbit such that, for almost every nonregular parameter $a$, the set of $n$ such that $\vert F^n(0;a) \vert < \delta_n$ is infinite. In particular, when $\delta_n = n^{-1}$, this extends an earlier result by Avila and Moreira.
\end{abstract}

\maketitle

\section{Introduction}

\subsection{Regular and nonregular parameters}
Given a real parameter \(a\), we let \(x \mapsto 1-ax^2 = F(x;a)\) denote the corresponding real quadratic map. We will study the recurrent behaviour of the critical point \(x = 0\) when the parameter belongs to the interval \([1,2]\). For such a choice of parameter there exists an invariant interval \(I_a \subset [-1,1]\), i.e.
\[
F(I_a;a) \subset I_a,
\]
containing the critical point \(x = 0\). The parameter interval is naturally divided into a \emph{regular} (\(\R\)) and \emph{nonregular} (\(\NR\)) part
\[
[1,2] = \R \cup \NR,
\]
with \(a \in \R\) being such that \(x \mapsto 1-ax^2\) has an attracting cycle, and \(\NR = [1,2] \smallsetminus \R\). These two sets turn out to be intertwined in an intricate manner, and this has led to an extensive study of the real quadratic family. We briefly mention some of the more fundamental results, and refer to \cite{Lyubich_notices} for an overview.

The regular maps are from a dynamic point of view well behaved, with almost every point, including the critical point, tending to the attracting cycle. This set of parameters, which with an application of the inverse function theorem is seen to be open, constitutes a large portion of $[1,2]$. The celebrated genericity result, known as the real Fatou conjecture, was settled independently by Graczyk--\' Swi\k{a}tek \cite{GS_fatou} and Lyubich \cite{Lyubich_fatou}: \emph{$\R$ is (open and) dense}. This has later been extended to real polynomials of arbitrary degree by Kozlovski--Shen--van Strien \cite{Koz_Shen_Strien}, solving the second part of the eleventh problem of Smale \cite{Smale}. The corresponding result for complex quadratic maps, the Fatou conjecture, is still to this day open.

The nonregular maps, in contrast to the regular ones, exhibit chaotic behaviour. In \cite{Jakobson} Jakobson showed the abundance of \emph{stochastic} maps, proving that the set of parameters $a \in \mathcal{S}$ for which the corresponding quadratic map has an absolutely continuous (with respect to Lebesgue) invariant measure (a.c.i.m), is of positive Lebesgue measure. This showed that, from a probabilistic point of view, nonregular maps are not negligible: for a regular map, any (finite) a.c.i.m is necessarily singular with respect to Lebesgue measure.

Chaotic dynamics is often associated with the notion of sensitive dependence on initial conditions. A compelling way to capture this property was introduced by Collet and Eckmann in \cite{Collet_Eckmann} where they studied certain maps of the interval having expansion along the critical orbit, proving abundance of chaotic behaviour. This condition is now known as the \emph{Collet--Eckmann condition}, and for a real quadratic map it states that
\begin{equation}\label{CE_condition}
\liminf_{n\to\infty} \frac{\log \vert \partial_x F^n(1;a)\vert}{n} > 0.
\end{equation}
Focusing on this condition, Benedicks and Carleson gave in their seminal papers \cite{BC1,BC2} another proof of Jakobson's theorem by proving the stronger result that the set $\CE$ of Collet--Eckmann parameters is of positive measure. As a matter of fact, subexponential increase of the derivative along the critical orbit is enough to imply the existence of an a.c.i.m, but the stronger Collet-Eckmann condition implies, and is sometimes equivalent with, ergodic properties such as exponential decay of correlations \cite{KN,Young,NS}, and stochastic stability \cite{BV}. 
For a survey on the role of the Collet--Eckmann condition in one-dimensional dynamics, we refer to \cite{CE_one_dimension}.

Further investigating the stochastic behaviour of nonregular maps, supported by the results in \cite{Lyubich_parapuzzles,MN}, Lyubich \cite{Lyubich_stochastic} established the following famous dichotomy: \emph{almost all real quadratic maps are either regular or stochastic}. Thus it turned out that the stochastic behaviour described by Jakobson is in fact typical for a nonregular map. In \cite{AM} Avila and Moreira later proved the strong result that expansion along the critical orbit is no exception either: \emph{almost all nonregular maps are Collet--Eckmann}. Thus a typical nonregular map have excellent ergodic properties.

\subsection{Recurrence and Theorem A}
In this paper we will study recurrence of the critical orbit to the critical point, for a typical nonregular (stochastic, Collet--Eckmann) real quadratic map. For this reason we introduce the following set.
\begin{Def}[Recurrence Set]\label{recurrence_set_definition}
Given a sequence $(\delta_n)_{n = 1}^\infty$ of real numbers, we define the \emph{recurrence set} as
\begin{align*}
\Lambda(\delta_n) = \{a \in \NR : \vert F^n(0;a)\vert < \delta_n\ \text{for finitely many}\ n\}.
\end{align*}
\end{Def}
In \cite{AM} Avila and Moreira also established the following recurrence result, proving a conjecture by Sinai: \emph{for almost every nonregular parameter $a$}
\[
\limsup_{n \to \infty} \frac{-\log \vert F^n(0;a) \vert}{\log n} = 1.
\]
Another way to state this result is as follows: for almost every nonregular parameter $a$, the set of $n$ such that $\vert F^n(0;a) \vert < n^{-\theta}$ is finite if $\theta > 1$ and infinite if $\theta < 1$. In terms of the above defined recurrence set, this result translates to
\[
\Leb \Lambda( n^{-\theta}) = \begin{cases} \Leb \NR &\mbox{if}\ \theta > 1,\\ 0 &\mbox{if}\ \theta < 1. \end{cases}
\]
In \cite{GS}, as a special case, a new proof of the positive measure case in the above stated result was obtained, together with a new proof that almost every nonregular map is Collet--Eckmann. In this paper we will give a new proof of the measure zero case, and in particular we will fill in the missing case of $\theta = 1$, thus completing the picture of polynomial recurrence. Our result will be restricted to the following class of recurrence rates.
\begin{Def}\label{admissible_definition}
A nonincreasing sequence $(\delta_n)$ of positive real numbers is called \emph{admissible} if there exists a constant $0 \leq \overline{e} < \infty$, and an integer $N \geq 1$, such that
\[
\delta_n \geq \frac{1}{n^{\overline{e}}} \qquad (n \geq N).
\]
\end{Def}
The following is the main result of this paper.
\begin{Thx}\label{Theorem_A}
There exists $\tau \in (0,1)$ such that if $(\delta_n)$ is admissible and
$$
\sum \frac{\delta_n}{\log n}\tau^{(\log^* n)^3} = \infty,
$$
then $\Leb (\Lambda(\delta_n) \cap \CE) = 0$.
\end{Thx}
Here $\log^*$ denotes to so-called \emph{iterated logarithm}, which is defined recursively as
$$
\log^* x = \begin{cases} 0 &\mbox{if}\ x \leq 1, \\ 1 + \log^* \log x &\mbox{if}\ x > 1. \end{cases}
$$
That is, $\log^* x$ is the number of times one has to iteratively apply the logarithm to $x$ in order for the result to be less than or equal to $1$. In particular, $\log^*$ grows slower than $\log_j = \log \circ \log_{j-1}$, for any $j \geq 1$.

Theorem A, together with the fact that almost every nonregular real quadratic map is Collet--Eckmann, clearly implies 
\begin{Cor}
$\Leb \Lambda(n^{-1}) = 0$.
\end{Cor}
\begin{Rem}
In fact, one can conclude the stronger statement 
\[
\Leb \Lambda(1/(n\log\log n)) = 0.
\]
At this moment we do not get any result for when $\delta_n = 1/(n \log n)$, and this would be interesting to investigate further.
\end{Rem}
One of the key points in the proof of Theorem~\ref{Theorem_A} is the introduction of unbounded distortion estimates; this differs from the classical Benedicks--Carleson techniques.

\begin{small}
\begin{acknowledgement}
This project has been carried out under supervision of Magnus Aspenberg as part of my doctoral thesis. I am very grateful to Magnus for proposing this problem, for his support, and for many valuable discussions and ideas. I express gratitude to my co-supervisor Tomas Persson for helpful comments and remarks. I would also like to thank Viviane Baladi for communicating useful references, and I thank Michael Benedicks for interesting discussions. Finally I thank the referee whose careful reading and comments helped improve the manuscript.
\end{acknowledgement}
\end{small}

\section{Reduction and outline of proof}

\subsection{Some definitions and Theorem B} 
We reduce the proof of Theorem~\ref{Theorem_A} to that of Theorem~\ref{Theorem_B}, stated below. For this we begin with some suitable definitions.

It will be convenient to explicitly express the constant in the Collet--Eckmann condition (\ref{CE_condition}), and for this reason we agree on the following definition.
\begin{Def}
Given $\gamma, C >0$ we call a parameter $a$ $(\gamma,C)$-Collet--Eckmann if 
\[
\vert \partial_x F^n(1;a) \vert \geq C e^{\gamma n} \qquad (n \geq 0).
\]
The set of all $(\gamma,C)$-Collet--Eckmann parameters is denoted $\CE(\gamma,C)$.
\end{Def}

Our parameter exclusion will be carried out on intervals centred at Collet--Eckmann parameters satisfying the following recurrence assumption.
\begin{Def}
A Collet--Eckmann parameter $a$ is said to have \emph{polynomial recurrence} (PR) if there exist constants $K=K(a) > 0$ and $\sigma = \sigma(a) \geq 0$ such that
\[
\vert F^n(0;a) \vert \geq \frac{K}{n^\sigma} \qquad (n \geq 1).
\]
The set of all PR-parameters is denoted $\PR$.
\end{Def}

Finally, we consider parameters for which the corresponding quadratic maps satisfy the reversed recurrence condition after some fixed time $N \geq 1$:
\[
\Lambda_N(\delta_n) = \{a \in \NR : \vert F^n(0;a) \vert \geq \delta_n\ \text{for all}\ n \geq N\}.
\]
Clearly we have that
\[
\Lambda(\delta_n) = \bigcup_{N \geq 1} \Lambda_N(\delta_n).
\]
Theorem~\ref{Theorem_A} will be deduced from
\begin{Thx}\label{Theorem_B}
There exists $\tau \in (0,1)$ such that if $(\delta_n)$ is admissible and 
\[
\sum \frac{\delta_n}{\log n}\tau^{(\log^* n)^3} = \infty,
\]
then for all $N \geq 1$, $\gamma > 0$, $C > 0$, and for all $a \in \mathcal{P}\mathcal{R}$, there exists an interval $\omega_a$ centred at $a$ such that
\[
\Leb(\Lambda_N(\delta_n) \cap \CE(\gamma,C) \cap \omega_a) = 0.
\]
\end{Thx}

\subsection{Proof of Theorem A}
Using Theorem~\ref{Theorem_B}, Theorem~\ref{Theorem_A} is proved by a standard covering argument. Since $\omega_a$ is centred at $a$, so is the smaller interval $\omega_a' = \omega_a/5$. By Vitali covering lemma there exists a countable collection $(a_j)$ of PR-parameters such that
\[
\PR \subset \bigcup_{a \in \PR} \omega'_a \subset \bigcup_{j = 1}^\infty \omega_{a_j}.
\]
It now follows directly that
\[
\Leb(\Lambda_N(\delta_n) \cap \CE(\gamma,C) \cap \PR) \leq \sum_{j=1}^\infty \Leb(\Lambda_N(\delta_n) \cap \CE(\gamma,C) \cap \omega_{a_j}) = 0,
\]
and therefore
\begin{align*}
\Leb (\Lambda(\delta_n) \cap \CE \cap \PR) &\leq \sum_{N,k,l \geq 1} \Leb(\Lambda_N(\delta_n) \cap \CE(k^{-1}\log 2,l^{-1}) \cap \PR) \\ 
&= 0.
\end{align*}
Finally, we notice that $\Lambda(\delta_n) \cap \CE \subset \PR$; indeed this is clearly the case since $(\delta_n)$ is assumed to be admissible.
\begin{Rem}
With the introduction of the set $\PR$ we are avoiding the use of previous recurrence results (e.g. Avila--Moreira) in order to prove Theorem~\ref{Theorem_A}, by (a priori) allowing $\PR$ to be a set of measure zero. In either case, the statement of Theorem~\ref{Theorem_A} is true.
\end{Rem}

\subsection{Outline of proof of Theorem B}
The proof of Theorem B will rely on the classical parameter exclusion techniques developed by Benedicks and Carleson \cite{BC1,BC2}, complemented with more recent results. In particular we allow for perturbation around a parameter in more general position than $a=2$. In contrast to the usual application of these techniques, our goal here is the show that what remains after excluding parameters is a set of zero Lebesgue measure. One of the key points in our approach is the introduction of unbounded distortion estimates.

We will carefully study the returns of the critical orbit, simultaneously for maps corresponding to parameters in a suitable interval $\omega \subset [1,2]$, to a small and fixed interval $(-\delta,\delta) = (-e^{-\Delta},e^{-\Delta})$. These returns to $(-\delta,\delta)$ will be classified as either \emph{inessential}, \emph{essential}, \emph{escape}, or \emph{complete}. Per definition of a complete return, we return close enough to $x = 0$ to be able to remove a large portion of $(-\delta_n,\delta_n)$ in phase space. To estimate what is removed in parameter space, we need distortion estimates. This will be achieved by i) enforcing a $(\gamma,C)$-Collet--Eckmann condition, and ii) continuously making suitable partitions in phase space: $(-\delta,\delta)$ is subdivided into partition elements $I_r = (e^{-r-1},e^{-r})$ for $r >0$, and $I_r = -I_{-r}$ for $r < 0$. Furthermore, each $I_r$ is subdivided into $r^2$ smaller intervals $I_{rl} \subset I_r$, of equal length $\vert I_r\vert/r^2$. After partitioning, we consider iterations of each partition element individually, and the proof of Theorem~\ref{Theorem_B} will be one by induction.

We make a few comments on the summability condition appearing in the statement of Theorem~\ref{Theorem_A} and Theorem~\ref{Theorem_B}. In order to prove our result we need to estimate how much is removed at a complete return, but also how long time it takes from one complete return to the next. The factor $\tau^{(\log^* n)^3}$ is connected to the estimate of what is removed at complete returns, and more specifically it is connected to distortion; as will be seen, our distortion estimates are unbounded. The factor $(\log n)^{-1}$ is directly connected to the time between two complete returns: if $n$ is the index of a complete return, it will take $\lesssim \log n$ iterations until we reach the next complete return.

In the next section we prove a couple of preliminary lemmas, and confirm the existence of a suitable start-up interval $\omega_a$ centred at $a \in \PR$, for which the parameter exclusion will be carried out. After that, the induction step will be proved, and an estimate for the measure of $\Lambda_N(\delta_n) \cap \CE(\gamma, C) \cap \omega_a$ will be given.

\section{Preliminary Lemmas}

In this section we establish three important lemmas that will be used in the induction step. These are derived from Lemma~2.6, Lemma~2.10, and Lemma~3.1 in \cite{Aspenberg}, respectively, where they are proved in the more general setting of a complex rational map.

\subsection{Outside Expansion Lemma}
The first result we will need is the following version of the classical Ma\~n\'e Hyperbolicity Theorem (see \cite{Melo_Strien}, for instance).
\begin{Lemma}[Outside Expansion]\label{outside_expansion_lemma}
Given a Collet--Eckmann parameter $a_0$ there exist constants $\gamma_M,C_M > 0$ such that, for all $\delta > 0$ sufficiently small, there is a constant $\epsilon_M = \epsilon_M(\delta) > 0$ such that, for all $a \in (a_0 - \epsilon_M,a_0 + \epsilon_M)$, if
\[
x, F(x;a), F^2(x;a), \dots, F^{n-1}(x;a) \notin (-\delta,\delta),
\]
then
\[
\vert \partial_x F^n(x;a) \vert \geq \delta C_M e^{\gamma_M n}.
\]
Furthermore, if we also have that $F^n(x;a) \in (-2\delta,2\delta)$, then
\[
\vert \partial_x F^n(x;a) \vert \geq C_M e^{\gamma_M n}.
\]
\end{Lemma}
A similar lemma for the quadratic family can be found in \cite{BBS} and \cite{Tsujii_positive}, for instance. The version stated here allows for $\delta$-independence at a more shallow return to the interval $(-2\delta,2\delta)$. To get this kind of annular result constitutes a minor modification of Lemma~4.1 in \cite{Tsujii_positive}. We refer to Lemma~2.6 in \cite{Aspenberg} and the proof therein, however, for a proof of the above result. This proof is based on Przytycki's \emph{telescope lemma} (see \cite{P4} and also \cite{PRS}). In contrast to the techniques in \cite{Tsujii_positive}, in the case of the quadratic family, no recurrence assumption is needed.

\subsection{Phase-parameter distortion}
If $t \mapsto F(x;a+t)$ is a family of (analytic) perturbations of $(x;a) \mapsto F(x;a)$ at $a$, we may expand each such perturbation as
\[
F(x;a+t) = F(x;a) + t\partial_aF(x;a) + \text{higher order terms},
\]
and it is easy to verify that 
\[
\frac{\partial_a F^n(x;a)}{\partial_x F^{n-1}(F(x;a);a)} = \frac{\partial_a F^{n-1}(x;a)}{\partial_x F^{n-2}(F(x;a);a)} + \frac{\partial_a F(F^{n-1}(x;a);a)}{\partial_x F^{n-1}(F(x;a);a)}.
\]
Our concern is with the quadratic family $x \mapsto 1-ax^2 = F(x;a)$, with $a$ being the parameter value. In particular we are interested in the critical orbit of each such member, and to this end we introduce the functions $a \mapsto \xi_j(a) = F^j(0;a)$, for $j \geq 0$. In view of our notation and the above relationship, we see that
\[
\frac{\partial_a F^n(0;a)}{\partial_x F^{n-1}(1;a)} = \sum_{k=0}^{n-1} \frac{\partial_aF(\xi_k(a);a)}{\partial_x F^k(1;a)}.
\]
Throughout the proof of Theorem~\ref{Theorem_B} it will be of importance to be able to compare phase and parameter derivatives. Under the assumption of exponential increase of the phase derivative along the critical orbit, this can be done, as is formulated in the following lemma. The proof is that of Lemma~2.10 in \cite{Aspenberg}.

\begin{Lemma}[Phase-Parameter Distortion]\label{tsujii_distortion}
Let $a_0$ be $(\gamma_0,C_0)$-Collet--Eckmann, $\gamma_T \in (0,\gamma_0)$, $C_T \in (0,C_0)$, and $A \in (0,1)$. There exist $T, N_T, \epsilon_T > 0$ such that if $a \in (a_0-\epsilon_T,a_0 + \epsilon_T)$ satisfies
$$
\vert \partial_x F^j(1;a) \vert \geq C_T e^{\gamma_T j} \qquad (j = 1,2,\dots, N_T,\dots n-1),
$$
for some $n-1 \geq N_T$, then
\[
(1-A)T \leq \left\vert \frac{\partial_a F^n(0;a)}{\partial_x F^{n-1} (1;a)} \right\vert \leq (1+A)T.
\]
\end{Lemma}
\begin{proof}
According to Theorem~3 in \cite{Tsujii_simple} (see also Theorem~1 in \cite{Levin})
\[
\lim_{j \to \infty} \frac{\partial_a F^j(0;a_0)}{\partial_x F^{j-1}(1;a_0)}  =  \sum_{k=0}^\infty \frac{\partial_a F(\xi_k(a_0);a_0)}{\partial_x F^k(1;a_0)}  = T \in \mathbb{R}_{>0}.
\]
Let $N_T > 0$ be large enough so that
\[
\left\vert \sum_{k=N_T}^\infty \frac{\partial_a F(\xi_k(a_0);a_0)}{\partial_x F^k(1;a_0)} \right\vert \leq \sum_{k=N_T}^\infty \frac{1}{C_0 e^{\gamma_0 k}} \leq \sum_{k=N_T}^\infty \frac{1}{C_T e^{\gamma_T k}} \leq \frac{1}{3}A T.
\]
Since $a \mapsto \partial_a F(\xi_k(a);a)/\partial_x F^k(1;a)$ is continuous there exists $\epsilon_T > 0$ such that given $a \in (a_0-\epsilon_T,a_0+\epsilon_T)$
\[
\left\vert \sum_{k=0}^{N_T-1} \frac{\partial_a F(\xi_k(a);a)}{\partial_x F(1;a)} - T \right\vert \leq \frac{1}{2}A T.
\]
Assuming $x \mapsto 1-ax^2$ to be $(\gamma_T,C_T)$-Collet--Eckmann up to time $n > N_T$, the result now follows since
\[
\left\vert \sum_{k=0}^n \frac{\partial_a F(\xi_k(a);a)}{\partial_x F^k(1;a)} - T \right\vert \leq AT.
\]
\end{proof}
\begin{Rem}
The quotient $(1+A)/(1-A) = D_A$ can be chosen arbitrarily close to $1$ by increasing $N_T$ and decreasing $\epsilon_T$.
\end{Rem}

\subsection{Start-up Lemma}
With the above two lemmas we now prove the existence of a suitable interval in parameter space on which the parameter exclusion will be carried out.

Given an admissible sequence $(\delta_n)$, let $N_A$ be the integer in Definition~\ref{admissible_definition}. Fix $N_B \geq 1, \gamma_B > 0$, and $C_B > 0$, and let $a_0$ be a PR-parameter satisfying a $(\gamma_0,C_0)$-Collet--Eckmann condition. In Lemma~\ref{tsujii_distortion} we make the choice 
\[
\gamma_T = \min(\gamma_B,\gamma_0,\gamma_M)/20  \quad \text{and} \quad C_T = \min(C_B,C_0)/3.
\]
Furthermore let
\[
\gamma = \min(\gamma_B,\gamma_0,\gamma_M)/2 \quad \text{and}\quad C = \min(C_B,C_0)/2,
\]
and let $m_{-1} = \max(N_A,N_B,N_T)$.

\begin{Lemma}[Start-up Lemma]\label{start-up_lemma}
There exist an interval $\omega_0 = (a_0-\epsilon,a_0 + \epsilon)$, an integer $m_0 \geq m_{-1}$, and a constant $S = \epsilon_1 \delta$ such that
\begin{enumerate}[(i)]
\item $\xi_{m_0} : \omega_0 \to [-1,1]$ is injective, and
\[
\vert \xi_{m_0}(\omega_0)\vert \geq \begin{cases} e^{-r}/r^2 &\mbox{if}\ \xi_{m_0}(\omega_0) \cap I_r \neq \emptyset, \\ S &\mbox{if}\ \xi_{m_0}\cap(-\delta,\delta) = \emptyset.\end{cases}
\]
\item Each $a \in \omega_0$ is $(\gamma,C)$-Collet--Eckmann up to time $m_0$:
\[
\vert \partial_x F^j(1;a) \vert \geq Ce^{\gamma j} \qquad (j=0,1,\dots,m_0-1).
\]
\item Each $a \in \omega_0$ enjoys polynomial recurrence up to time $m_0$: there exist absolute constants $K > 0$ and $\sigma \geq 0$ such that for $a \in \omega_0$
\[
\vert \xi_j(a) \vert \geq \frac{K}{j^\sigma} \qquad (j = 1,2,\dots,m_0-1).
\]

\end{enumerate}
\end{Lemma}

\begin{proof} 
Given $x,y \in \xi_{j}(\omega_0)$, $j \geq 1$, consider the following distance condition
\begin{equation}\label{dist_condition}
\vert x - y \vert \leq \begin{cases} e^{-r}/r^2 &\mbox{if}\ \xi_{j}(\omega_0) \cap I_r \neq \emptyset, \\ S = \epsilon_1 \delta &\mbox{if}\ \xi_{j}(\omega_0) \cap (-\delta,\delta) = \emptyset. \end{cases}
\end{equation}
By making $\epsilon$ smaller, we may assume that (\ref{dist_condition}) is satisfied up to time $m_{-1}$. Moreover, we make sure that $\epsilon$ is small enough to comply with Lemma~\ref{tsujii_distortion}. Whenever (\ref{dist_condition}) is satisfied, phase derivatives are comparable as follows
\begin{equation}\label{phaseDist}
\frac{1}{C_1} \leq \left\vert \frac{\partial_x F(x;a)}{\partial_x F(y;b)} \right\vert \leq C_1,
\end{equation}
with $C_1 > 1$ a constant. This can be seen through the following estimate
\[
\left\vert \frac{\partial_x F(x;a)}{\partial_x F(y;b)}\right\vert = \left\vert \frac{-2ax}{-2by} \right\vert \leq \frac{a_0 + \epsilon}{a_0-\epsilon}\left(\left\vert \frac{x-y}{y} \right\vert + 1\right).
\]
If we are outside $(-\delta,\delta)$ then
\[
\left\vert \frac{x-y}{y} \right\vert \leq \frac{S}{\delta} = \epsilon_1,
\]
and if we are hitting $I_r$ with largest possible $r$,
\[
\left\vert \frac{x-y}{y} \right\vert \leq \frac{e^{-r}}{r^2}\frac{1}{e^{-(r+1)}} = \frac{e}{r^2} \leq \frac{e}{\Delta^2}.
\]
By making sure that $\epsilon$, $\epsilon_1$, and $\delta$ are small enough, $C_1$ can be made as close to $1$ as we want. In particular, we make $C_1$ close enough to $1$ so that
\begin{equation}\label{CEafterN}
C_1^{-j}C_0 e^{\gamma_0 j} \geq C e^{\gamma j} \qquad (j \geq 0).
\end{equation}

As long as the distance condition (\ref{dist_condition}) is satisfied, we will have good expansion along the critical orbits. Indeed by (\ref{phaseDist}) and (\ref{CEafterN}) it follows that, given $a \in \omega_0$,
\begin{align*}
\vert \partial_x F^j(1;a) \vert &\geq C_1^{-j}\vert \partial_x F^j(1;a_0) \vert \\
&\geq C_1^{-j}C_0 e^{\gamma_0 j} \\
&\geq C e^{\gamma j} \qquad (j \geq 0\ \text{such that}\ (2)\ \text{is satisfied}).
\end{align*}
This tells us that, during the time for which (\ref{dist_condition}) is satisfied, each $a \in \omega_0$ is $(\gamma,C)$-Collet--Eckmann. In particular, since $\gamma > \gamma_T$ and $C > C_T$, we can apply Lemma~\ref{tsujii_distortion}, and together with the mean value theorem we have that
\begin{align*}
\vert \xi_{j}(\omega_0)\vert &= \vert \partial_a F^j(0;a') \vert \vert \omega_0 \vert \\
&\geq (1-A)T \vert \partial_x F^{j-1}(1;a') \vert \vert \omega_0\vert \\
&\geq (1-A)T C e^{\gamma (j-1)} \vert \omega_0 \vert.
\end{align*}
Our interval is thus expanding, and we let $m_0 = j$, with $j \geq m_{-1}$ the smallest integer for which (\ref{dist_condition}) is no longer satisfied. This proves statements (i) and (ii).

To prove statement (iii), let $K_0 > 0$ and $\sigma_0 \geq 0$ be the constants associated to $a_0$ for which
\[
\vert \xi_j(a_0) \vert \geq \frac{K_0}{j^{\sigma_0}} \qquad (j \geq 1).
\]
In view of (\ref{dist_condition}), when we hit $(-\delta,\delta)$ at some time $j < m_0$,
\[
\vert \xi_j(a) \vert \geq \vert \xi_j(a_0) \vert - \vert \xi_j(\omega_0) \vert \geq \vert \xi_j(a_0) \vert - \frac{e^{-r}}{r^2}.
\]
Here, $r$ is such that
\[
e^{-r-1} \leq \vert \xi_j(a_0) \vert,
\]
and therefore, given $\delta$ small enough,
\[
\vert \xi_j(a) \vert \geq \vert \xi_j(a_0) \vert \left(1 - \frac{e}{\Delta^2} \right) \geq \frac{K_0/2}{j^{\sigma_0}} \qquad (j = 1,2,\dots,m_0-1).
\]
\end{proof}
\begin{Rem}\label{large_scale}
By making $\delta$ small enough so that $1/\Delta^2 < \epsilon_1$, $S$ will be larger than any partition element $I_{rl} \subset (-\delta,\delta)$. This $S$ is usually referred to as the \emph{large scale}.
\end{Rem}
Since $\Lambda_{N_B}(\delta_n) \subset \Lambda_{m_0}(\delta_n)$, Theorem~\ref{Theorem_B} follows if
\[
\Leb\left(\Lambda_{m_0}(\delta_n) \cap \CE(\gamma_B,C_B) \cap \omega_0\right) = 0.
\]

\section{Induction Step}

\subsection{Initial iterates}
Let $\omega_0 = \Delta_0$ be the start-up interval obtained in Lemma \ref{start-up_lemma}. Iterating this interval under $\xi$ and successively excluding parameters that do not satisfy the recurrence condition, or the Collet--Eckmann condition, we will inductively define a nested sequence $\Delta_0 \supset \Delta_1 \supset \dots \supset \Delta_k \supset \cdots$ of sets of parameters satisfying
\[
\Lambda_{m_0}(\delta_n) \cap \CE(\gamma_B,C_B) \cap \omega_0 \subset \Delta_\infty = \bigcap_{k = 0}^\infty \Delta_k,
\]
and our goal is to estimate the Lebesgue measure of $\Delta_\infty$. This will require a careful analysis of the so-called \emph{returns} to $(-\delta,\delta)$, and we will distinguish between four types of returns: \emph{inessential}, \emph{essential}, \emph{escape}, and \emph{complete}. At the $k^\text{th}$ complete return, we will be in the position of excluding parameters and form the partition that will make up the set $\Delta_k$. Below we will describe the iterations from the $k^\text{th}$ complete return to the $(k+1)^\text{th}$ complete return, hence the forming of $\Delta_{k+1}$. Before indicating the partition, and giving a definition of the different returns, we begin with considering the first initial iterates of $\xi_{m_0}(\omega_0)$.

If $\xi_{m_0}(\omega_0) \cap (-\delta,\delta) \neq \emptyset$, then we have reached a return and we proceed accordingly as is described below. If this is not the case, then we are in the situation
\[
\xi_{m_0}(\omega_0) \cap (-\delta,\delta) = \emptyset \quad \text{and}\quad \vert \xi_{m_0}(\omega_0) \vert \geq S,
\]
with $S$ larger than any partition element $I_{rl} \subset (-\delta,\delta)$ (see Remark~\ref{large_scale}). Since the length of the image is bounded from below, there is an integer $n^* = n^*(S)$ such that for some smallest $n \leq n^*$ we have
\[
\xi_{m_0+n}(\omega_0) \cap (-\delta,\delta) \neq \emptyset.
\]
In this case, $m_0 + n$ is the index of the first return. We claim that, if $m_0$ is large enough, we can assume good derivative up to time $m_0 + n$. To realise this, consider for $j < n$ the distortion quotient
\[
\left\vert\frac{\partial_x F^{m_0+j}(1;a)}{\partial_x F^{m_0+j}(1;b)} \right\vert = \left\vert\frac{\partial_x F^{m_0-1}(1;a)}{\partial_x F^{m_0-1}(1;b)} \right\vert \left\vert\frac{\partial_x F^{j+1}(\xi_{m_0}(a);a)}{\partial_x F^{j+1}(\xi_{m_0}(b);b)} \right\vert.
\]
Since the distance conditions (\ref{dist_condition}) are satisfied up to time $m_0 - 1$, the first factor in the above right hand side is bounded from above by the constant $C_1^{m_0-1}$, with $C_1 > 1$ being very close to $1$ (see (\ref{phaseDist})). Furthermore, since $j < n < n^*(S)$, and since we by assumption are iterating outside $(-\delta,\delta)$, the second factor in the above right hand side is bounded from above by some positive constant $C_{S,\delta}$ dependent on $S$ and $\delta$.

If there is no parameter $a' \in \omega_0$ such that $\vert \partial_x F^{m_0+j}(1;a') \vert \geq C_B e^{\gamma_B(m_0 + j)}$ then we have already reached our desired result. If on the other hand there is such a parameter $a'$ then for all $a \in \omega_0$ it follows from the above distortion estimate and our choice of $\gamma$ that
\[
\vert \partial_x F^{m_0+j}(1;a) \vert \geq \frac{C_B e^{\gamma_B (m_0+j)}}{C_1^{m_0-1}C_{S,\delta}} \geq Ce^{\gamma(m_0+j)},
\]
provided $m_0$ is large enough. We conclude that
\begin{equation}\label{initial_derivative}
\vert \partial_x F^{j}(1;a) \vert \geq C e^{\gamma j} \qquad (a \in \omega_0,\ j = 0,1,\dots,m_0+n-1).
\end{equation}

In the case we have to iterate $\xi_{m_0}(\omega_0)$ further to hit $(-\delta,\delta)$ we still let $m_0$ denote the index of the first return.

\subsection{The partition}
At the $(k+1)^\text{th}$ step in our process of excluding parameters, $\Delta_k$ consists of disjoint intervals $\omega_k^{rl}$, and for each such interval there is an associated time $m_k^{rl}$ for which either $\xi_{m_k^{rl}}(\omega_k^{rl}) = I_{rl} \subset (-4\delta,4\delta)$, or $\xi_{m_k^{rl}}(\omega_k^{rl})$ is mapped onto $\pm(\delta,x)$, with $\vert x-\delta \vert \geq 3\delta$. We iterate each such interval individually, and let $m_{k+1}^{rl}$ be the time for which $\xi_{m_{k+1}^{rl}}(\omega_k^{rl})$ hits deep enough for us to be able to remove a significant portion of $(-\delta_{m_{k+1}^{rl}},\delta_{m_{k+1}^{rl}})$ in phase space, and let $E_k^{rl}$ denote the corresponding set that is removed in parameter space. We now form the set $\hat{\omega}_k^{rl} \subset \Delta_{k+1}$ and make the partition
\[
\hat{\omega}_k^{rl} = \omega_k^{rl} \smallsetminus E_k^{rl} = \left( \bigcup_{r',l'} \omega_{k+1}^{r'l'} \right) \cup T_{k+1} = N_{k+1} \cup T_{k+1}.
\]
Here, each $\omega_{k+1}^{r'l'} \subset N_{k+1}$ is such that $\xi_{m_{k+1}^{rl}}(\omega_{k+1}^{r'l'}) = I_{r'l'} \subset (-4\delta,4\delta)$, and $T_{k+1}$ consists of (at most) two intervals whose image under $\xi_{m_{k+1}^{rl}}$ is $\pm (\delta,x)$, with $\vert x-\delta \vert \geq 3 \delta$.
\begin{Rem}
At most four intervals $\omega_{k+1}^{r'l'} \subset N_{k+1}$ will be mapped onto an interval slightly larger than $I_{r'l'}$, i.e.
\[
I_{r'l'} \subset \xi_{m_{k+1}^{rl}}(\omega_{k+1}^{r'l'}) \subset I_{r'l'}\cup I_{r''l''},
\]
with $I_{r'l'}$ and $I_{r''l''}$ adjacent partition elements.
\end{Rem}
\begin{Rem}
At \emph{essential returns} and \emph{escape returns} we will, if possible, make a partial partition. To these partitioned parameter intervals we associate a \emph{complete return time} even though nothing is removed at these times. This is described in more detail in sections 4.8 and 4.9.
\end{Rem}
\begin{Rem}
Notice that our way of partitioning differs slightly from the original one considered in \cite{BC1}, since here we do not continue to iterate what is mapped outside of $(-\delta,\delta)$, but instead stop and make a partition.
\end{Rem}

\subsection{The different returns to \texorpdfstring{$(-\delta,\delta)$}{(-delta,delta)}}
At time $m_{k+1}^{rl}$ we say that $\omega_k^{rl}$ has reached the $(k+1)^\text{th}$ \emph{complete return} to $(-\delta,\delta)$. In between the two complete returns of index $m_{k}^{rl}$ and $m_{k+1}^{rl}$ we might have returns which are not complete. Given a return at time $n > m_k^{rl}$, we classify it as follows.
\begin{enumerate}[i)]
\item If $\xi_n(\omega_k^{rl}) \subset I_{r'l'} \cup I_{r''l''}$, with $I_{r'l'}$ and $I_{r''l''}$ adjacent partition elements ($r' \geq r''$), and if $\vert \xi_n(\omega_k^{rl}) \vert < \vert I_{r'l'}\vert$, we call this an \emph{inessential return}. The interval $I_{r'l'}\cup I_{r''l''}$ is called the \emph{host interval}.
\item If the return is not inessential, it is called an \emph{essential return}. The outer most partition element $I_r$ contained in the image is called the \emph{essential interval}.
\item If $\xi_n(\omega_k^{rl}) \cap (-\delta,\delta) \neq \emptyset$ and $\vert \xi_n(\omega_k^{rl}) \smallsetminus (-\delta,\delta) \vert \geq 3\delta$, we call this an \emph{escape return}. The interval $\xi_n(\omega_k^{rl}) \smallsetminus (-\delta,\delta)$ is called the \emph{escape interval}.
\item Finally, if a return satisfies $\xi_n(\omega_k^{rl}) \cap (-\delta_n/3,\delta_n/3) \neq \emptyset$, it is called a \emph{complete return}.
\end{enumerate}
We use these terms exclusively, that is, an inessential return is not essential, an essential return is not an escape, and an escape return is not complete.

Given $\omega_k^{rl} \subset \Delta_k$ we want to find an upper bound for the index of the next complete return. In the worst case scenario we encounter all of the above kind of returns, in the order
\[
\text{complete} \to \text{inessential} \to \text{essential} \to \text{escape} \to \text{complete}.
\]
Given such behaviour, we show below that there is an absolute constant $\kappa>0$ such that the index of the $(k+1)^{\text{th}}$ complete return satisfies $m_{k+1}^{rl} \leq m_k^{rl} + \kappa \log m_k^{rl}$.

\subsection{Induction assumptions}
Up until the start time $m_0$ we do not want to assume anything regarding recurrence with respect to our recurrence rate $(\delta_n)$. Since the perturbation is made around a PR-parameter $a_0$, we do however have the following \emph{polynomial recurrence} to rely on (Lemma~\ref{start-up_lemma}):
\begin{enumerate}
\item[(PR)] $\vert F^j(0;a) \vert \geq K/j^\sigma$ for all $a \in \omega^{rl}_k$ and $j = 1,2,\dots,m_0-1$.
\end{enumerate}
After $m_0$ we start excluding parameters according to the following \emph{basic assumption}:
\begin{enumerate}
\item[(BA)] $\vert F^j(0;a) \vert \geq \delta_j/3$ for all $a \in \omega_k^{rl}$ and $j = m_0,m_0+1,\dots,m_k^{rl}$.
\end{enumerate}
Since our sequence $\delta_j$ is assumed to be admissible, we will frequently use the fact that $\delta_j/3 \geq 1/(3j^{\overline{e}})$.

From (\ref{initial_derivative}) we know that every $a \in \omega^{rl}_k$ is $(\gamma,C)$-Collet--Eckmann up to time $m_0$, and this condition is strong enough to ensure phase-parameter distortion (Lemma~\ref{tsujii_distortion}). We will continue to assume this condition at complete returns, but in between two complete returns we will allow the exponent to drop slightly due to the loss of derivative when returning close to the critical point $x = 0$. We define the \emph{basic exponent} conditions as follows:
\begin{enumerate}
\item[(BE)(1)] $\vert \partial_x F^{m_k^{rl}-1}(1;a) \vert \geq C e^{\gamma (m_k^{rl}-1)}$ for all $a \in \omega^{rl}_k $.

\item[(BE)(2)] $\vert \partial_x F^j(1;a) \vert \geq C e^{(\gamma/3) j}$ for all $a \in \omega^{rl}_k$ and $j = 0,1,\dots,m_k^{rl}-1$.
\end{enumerate}

Assuming (BA) and (BE)(1,2) for $a \in \omega_k^{rl} \subset \Delta_k$, we will prove it for $a' \in \omega_{k+1}^{r'l'} \subset \Delta_{k+1} \subset \Delta_k$. Before considering the iteration of $\omega_k^{rl}$, we define the \emph{bound period} and the \emph{free period}, and prove  some useful lemmas connected to them. For technical reasons these lemmas will be proved using the following weaker assumption on the derivative. Given a time $n \geq m^{rl}_k$ we consider the following condition:
\begin{enumerate}
\item[(BE)(3)] $\vert \partial_x F^j(1;a) \vert \geq C e^{(\gamma/9)j}$ for all $a \in \omega_k^{rl}$ and $j = 0,1,\dots,n-1$.
\end{enumerate}
Notice that $\gamma/9 > \gamma_T$, hence we will be able to apply Lemma~\ref{tsujii_distortion} at all times.

To rid ourselves of cumbersome notation we drop the indices from this point on and write $\omega = \omega_k^{rl}$, and $m = m_k^{rl}$.

\subsection{The bound and free periods}
Assuming we are in the situation of a return for which $\xi_n(\omega) \subset I_{r+1} \cup I_r \cup I_{r-1} \subset (-4\delta,4\delta)$, we are relatively close to the critical point, and therefore the next iterates $\xi_{n+j}(\omega)$ will closely resemble those of $\xi_j(\omega)$. We quantify this and define the \emph{bound period} associated to this return as the maximal $p$ such that
\begin{enumerate}
\item[(BC)] $\vert \xi_\nu(a) - F^\nu(\eta;a) \vert \leq \vert \xi_\nu(a) \vert/(10\nu^2)$ for $\nu = 1,2,\dots,p$
\end{enumerate}
holds for all $a \in \omega$, and all $\eta \in (0,e^{-\vert r-1\vert})$. We refer to (BC) as the \emph{binding condition}.
\begin{Rem}\label{pointwise_binding}
In the proof of Lemma~\ref{window_lemma} we will refer to \emph{pointwise binding}, meaning that for a given parameter $a$ we associate a bound period $p = p(a)$ according to when (BC) breaks for this specific parameter. We notice that the conclusions of Lemma~\ref{bound_distortion} and Lemma~\ref{bound_period_lemma} below are still true if we only consider iterations of one specific parameter.
\end{Rem}

The bound period is of central importance, and we establish some results connected to it (compare with \cite{BC1}). An important fact is that during this period the derivatives are comparable in the following sense.
\begin{Lemma}[Bound distortion]\label{bound_distortion}
Let $n$ be the index of a return for which $\xi_n(\omega) \subset I_{r+1} \cup I_r \cup I_{r-1}$, and let $p$ be the bound period. Then, for all $a \in \omega$ and $\eta \in (0,e^{-\vert r-1\vert})$,
\[
\frac{1}{2} \leq \left\vert \frac{\partial_x F^j(1-a\eta^2;a)}{\partial_xF^j(1;a)} \right\vert \leq 2 \qquad (j = 1,2,\dots,p).
\]
\end{Lemma}
\begin{proof}
It is enough to prove that
\begin{equation}\label{b_dist}
\left \vert \frac{\partial_x F^j(1-a\eta^2;a)}{\partial_x F^j(1;a)} - 1 \right\vert \leq \frac{1}{2}.
\end{equation}
The quotient can be expressed as
\[
\frac{\partial_x F^j(1-a\eta^2;a)}{\partial_x F^j(1;a)} = \prod_{\nu=1}^j \left( \frac{F^\nu(\eta;a) - \xi_\nu(a)}{\xi_\nu(a)} + 1\right),
\]
and applying the elementary inequality
\[
\left\vert \prod_{\nu = 1}^j (u_n + 1) - 1\right\vert \leq \exp \left(\sum_{\nu=1}^j \vert u_n \vert \right) -1,
\]
valid for complex $u_n$, (\ref{b_dist}) now follows since
\[
\sum_{\nu = 1}^j \frac{\vert F^\nu(\eta;a) - \xi_\nu(a)\vert}{\vert \xi_\nu(a) \vert} \leq \frac{1}{10}\sum_{\nu=1}^j \frac{1}{\nu^2} \leq \log \frac{3}{2}.
\]
\end{proof}
The next result gives us an estimate of the length of the bound period. As will be seen, if (BA) and (BE)(3) are assumed up to time $n \geq m = m_k^{rl}$, the bound period is never longer than $n$, and we are therefore allowed to use the induction assumptions during this period. In particular, in view of the above distortion result and (BE)(3), we inherit expansion along the critical orbit during the bound period; making sure $m_0$ is large enough, and using (BA) together with the assumption that $(\delta_n)$ is admissible, we have
\begin{align}\label{derivative_after_bound_period}
\vert \partial_x F^{n + j}(1;a) \vert &= 2a \vert \xi_n(a) \vert \vert \partial_x F^{n-1}(1;a) \vert \vert \partial_x F^j(1-a\xi_n(a)^2;a) \vert \nonumber \\
&\geq \frac{2}{3n^{\overline{e}}} C^2e^{(\gamma/9)(n+j-1)} \nonumber \\
&= \frac{2}{3}C^2 e^{-\gamma/9} \exp\left\{ \left(\frac{\gamma}{9} - \frac{\overline{e}\log n}{n+j}\right)(n+j)\right\} \nonumber \\
&\geq C_T e^{\gamma_T (n+j)} \qquad (j = 0,1,\dots,p).
\end{align}
This above estimate is an a priori one, and will allow us to use Lemma~\ref{tsujii_distortion} in the proof of Lemma~\ref{free_period_lemma}.

\begin{Lemma}[Bound Length]\label{bound_period_lemma}
Let $n$ be the index of a return such that $\xi_n(\omega) \subset I_{r+1} \cup I_r \cup I_{r-1}$, and suppose that (BA) and (BE)(3) are satisfied up to time $n$. Then there exists a constant $\kappa_1 > 0$ such that the corresponding bound period satisfy
\begin{equation}\label{bound_period}
\kappa_1^{-1}r \leq p \leq \kappa_1 r.
\end{equation}
\end{Lemma}
\begin{proof}
By the mean value theorem and Lemma \ref{bound_distortion} we have that
\begin{align}\label{bound_estimate}
\vert \xi_j(a) - F^j(\eta;a) \vert &= \vert F^{j - 1}(1;a) - F^{j-1}(1-a\eta^2;a) \vert \nonumber \\
&= a\eta^2 \vert \partial_x F^{j-1}(1-a\eta'^2;a)\vert \\
&\geq \frac{a\eta^2}{2} \vert \partial_x F^{j-1}(1;a) \vert, \nonumber
\end{align}
as long as $j \leq p$. (Here, $0 < \eta' < \eta$.) Furthermore, as long as we also have $j \leq (\log n)^2$, say, we can use the induction assumptions: using (BE)(3) we find that
\[
\frac{1}{2} e^{-2(r+1)} C e^{(\gamma/9)(j-1)} \leq \frac{a\eta^2}{2} \vert \partial_x F^{j-1}(1;a)\vert \leq \frac{\vert \xi_j(a)\vert}{10j^2} \leq 1.
\]
Taking the logarithm, using (BA), and making sure that $m_0$ is large enough, we therefore have
\[
j \leq 1 + \frac{9}{\gamma}\left(2r + 2 + \log 2 - \log C\right) \lesssim r \lesssim \log n \leq (\log n)^2,
\]
as long as $j \leq p$ \emph{and} $j \leq (\log n)^2$. This tells us that $j \leq p$ must break before $j \leq (\log n)^2$; in particular there is a constant $\kappa_1 > 0$ such that $p \leq \kappa_1 r$.

For the lower bound consider $j = p+1$ and the equality (\ref{bound_estimate}). With $a \in \omega$ being the parameter for which the inequality in the binding condition is reversed, using Lemma \ref{bound_distortion} we find that
\[
\frac{\vert \xi_{p+1}(a)\vert}{10(p+1)^2} \leq \vert \xi_{p+1}(a) - F^{p+1}(\eta;a)\vert \leq 4e^{-2r}\vert \partial_x F^p(1;a) \vert \leq 4e^{-2r} 4^p.
\]
Using the upper bound for $p$ we know that (BA) (or (PR)) is valid at time $p+1$, hence
\[
\frac{\vert \xi_{p+1}(a) \vert}{10(p+1)^2} \geq \frac{1}{30(p+1)^{2+\hat{e}}},
\]
where $\hat{e} = \max(\overline{e},\sigma)$. Therefore
\[
\frac{1}{30(p+1)^{2 + \hat{e}}} \leq 4 e^{-2r}4^p,
\]
and taking the logarithm proves the lower bound.

\end{proof}
\begin{Rem}\label{lower_bound_bound_period_remark}
Notice that the lower bound is true without assuming the upper bound (which in our proof requires (BE)(3) at time $n$) as long as we assume (BA) to hold at time $p+1$.
\end{Rem}

The next result will concern the growth of $\xi_n(\omega)$ during the bound period.
 
\begin{Lemma}[Bound Growth]\label{bound_growth_lemma}
Let $n$ be the index of a return such that $\xi_n(\omega) \subset I$ with $I_{rl} \subset I \subset I_{r+1} \cup I_r \cup I_{r-1}$, and suppose that (BA) and (BE)(3) are satisfied up to time $n$. Then there exists a constant $\kappa_2 > 0$ such that
\[
\vert \xi_{n+p+1}(\omega) \vert \geq \frac{1}{r^{\kappa_2}}\frac{\vert\xi_n(\omega) \vert}{\vert I \vert}.
\]
\end{Lemma}
\begin{proof}
Denote $\Omega = \xi_{n+p+1}(\omega) $ and notice that for any two given parameters $a,b \in \omega$ we have
\begin{align}\label{bound_length_step1}
\vert \Omega \vert &\geq \vert F^{n + p + 1}(0;a) - F^{n + p + 1}(0;b) \vert \nonumber \\
&= \vert F^{p + 1}(\xi_{n}(a);a) - F^{p + 1}(\xi_{n}(b);b) \vert \nonumber \\
&\geq \vert F^{p+1}(\xi_{n}(a);a) - F^{p+1}(\xi_{n}(b);a) \vert \nonumber \\
&\qquad - \vert F^{p+1}(\xi_{n}(b);a) - F^{p+1}(\xi_{n}(b);b)\vert.
\end{align}
Due to exponential increase of the phase derivative along the critical orbit, the dependence on parameter is inessential in the following sense:
\begin{equation}\label{parameter_independence}
\vert F^{p+1}(\xi_{n}(b);a) - F^{p+1}(\xi_{n}(b);b)\vert \leq e^{-(\gamma /18)n} \vert \xi_n(\omega)\vert.
\end{equation}
To realise this, first notice that we have the following (somewhat crude) estimate for the parameter derivative:
\[
\vert \partial_a F^j(x;a) \vert \leq 5^j \qquad (j=1,2,\dots).
\]
Indeed, $\vert \partial_a F(x;a) \vert \leq 1 < 5$, and by induction
\begin{align*}
\vert \partial_a F^{j+1}(x;a)\vert &= \vert \partial_a (1-aF^j(x;a)^2) \vert \\
&= \vert -F^j(x;a)^2 - 2aF^j(x;a)\partial_a F^j(x;a) \vert \\
&\leq 1 + 4\cdot5^j \\
&\leq 5^{j+1}.
\end{align*}
Using the mean value theorem twice, Lemma \ref{tsujii_distortion} and (BE)(3) we find that
\[
\vert F^{p+1}(\xi_{n}(b);a) - F^{p+1}(\xi_{n}(b);b)\vert \leq [(1-A)T]^{-1} 5^{p+1}C^{-1}e^{-(\gamma/9) (n-1)}\vert \xi_n(\omega)\vert.
\]
In view of (\ref{bound_period}) and (BA), making $m_0$ larger if needed, the inequality (\ref{parameter_independence}) can be achieved.

Assume now that at time $p+1$ (BC) is broken for parameter $a$, and let $b$ be an endpoint of $\omega$ such that
\[
\vert \xi_n(a) - \xi_n(b) \vert \geq \frac{\vert \xi_n(\omega) \vert}{2}.
\]
Continuing the estimate of $\vert \Omega \vert$, using (\ref{parameter_independence}), we find that
\begin{align}\label{bound_length_step2}
\vert \Omega \vert &\geq \vert F^{p}(1-a\xi_{n}(a)^2;a) - F^{p}(1-a\xi_{n}(b)^2;a)\vert \nonumber \\ 
&\qquad - \vert F^{p+1}(\xi_{n}(b);a) - F^{p+1}(\xi_{n}(b);b)\vert \nonumber \\
&\geq \left(a \vert \xi_{n}(a) + \xi_{n}(b)\vert \vert \partial_x F^{p}(1-a\xi_{n}(a')^2;a)\vert - 2e^{-(\gamma /18)n}\right) \frac{\vert \xi_n(\omega) \vert}{2} \nonumber \\
&\geq \left(2ae^{-r}\vert \partial_xF^{p}(1-a\xi_{n}(a')^2;a) \vert - 2e^{-(\gamma /18)n} \right) \frac{\vert \xi_n(\omega)\vert}{2}.
\end{align}
Using Lemma \ref{bound_distortion} twice and the equality in (\ref{bound_estimate}) (with $p+1$ instead of $p$) together with (BC) (now reversed inequality) we continue the estimate in (\ref{bound_length_step2}) to find that
\begin{align}\label{bound_length_step3}
\vert \Omega \vert &\geq \left( 2ae^{-r}\frac{1}{4 a\eta^2}\frac{\vert \xi_{p+1} (a) \vert}{10(p+1)^2} - 2e^{-(\gamma/18)n} \right) \frac{\vert \xi_n(\omega)\vert}{2} \nonumber \\
&\geq \left( e^r \frac{\vert \xi_{p+1}(a) \vert}{20 (p+1)^2} - 2e^{-(\gamma/18)n} \right) \frac{\vert \xi_n(\omega)\vert}{2}.
\end{align}
In either case of $p \leq m_0$ or $p > m_0$ we have that (using (BA), (PR), and the assumption that our recurrence rate is admissible)
\[
\frac{\vert \xi_{p+1}(a) \vert}{(p+1)^2} \geq \frac{K}{3(p+1)^{2+\hat{e}}},
\]
where $\hat{e} = \max(\overline{e},\sigma)$.
We can make sure that the second term in the parenthesis in (\ref{bound_length_step3}) is always less than a fraction, say $1/2$, of the first term and therefore, using (BC), (\ref{bound_period}), and that $e^r \geq 1/(2r^2 \vert I \vert)$, we finish the estimate as follows
\begin{align}\label{bound_length_final}
\vert \Omega \vert &\geq \frac{K}{240}\frac{1}{(p+1)^{2+\hat{e}}} \vert \xi_n(\omega)\vert e^r \nonumber \\
&\geq \frac{K}{480}\frac{1}{r^2 (p+1)^{2 + \hat{e}}} \frac{\vert \xi_n(\omega) \vert}{\vert I \vert}\nonumber \\
&\geq \frac{K}{480 (2\kappa_1)^{2+\hat{e}}}\frac{1}{r^{4+\hat{e}}} \frac{\vert \xi_n(\omega) \vert}{\vert I \vert} \nonumber \\
&\geq \frac{1}{r^{\kappa_2}} \frac{\vert \xi_n(\omega) \vert}{\vert I \vert},
\end{align}
where we can choose $\kappa_2 = 5 + \hat{e}$ as long as $\delta$ is sufficiently small.
\end{proof}
\begin{Rem}
Using the lower bound for $p$, the upper bound
\[
\vert \xi_{n+p+1}(\omega) \vert \leq \frac{1}{r} \frac{\vert \xi_n(\omega)\vert}{\vert I \vert}
\]
can be proved similarly.
\end{Rem}

This finishes the analysis of the bound period, and we continue with describing the \emph{free period}. 
A free period will always follow a bound period, and during this period we will be iterating outside $(-\delta,\delta)$. We let $L$ denote the length of this period, i.e. $L$ is the smallest integer for which
\[
\xi_{n+p+L}(\omega) \cap (-\delta,\delta) \neq \emptyset.
\]
The following lemma gives an upper bound for the length of the free period, following the bound period of a complete return, or an essential return.
\begin{Lemma}[Free length]\label{free_period_lemma}
Let $\xi_n(\omega) \subset I_{r+1} \cup I_r \cup I_{r-1}$ with $n$ being the index of a complete return or an essential return, and suppose that (BA) and (BE)(3) are satisfied up to time $n$. Let $p$ be the associated bound period, and let $L$ be the free period. Then there exists a constant $\kappa_3 > 0$ such that
\[
L \leq \kappa_3 r.
\]
\end{Lemma}
\begin{proof}
Assuming $j \leq L$ and $j \leq (\log n)^2$, similar calculations as in the proof of Lemma \ref{bound_growth_lemma} gives us parameter independence (see (\ref{parameter_independence}) and notice that from (\ref{derivative_after_bound_period}) we are allowed to apply Lemma \ref{tsujii_distortion}); using Lemma \ref{bound_growth_lemma} and Lemma \ref{outside_expansion_lemma} we find that
\[
2 \geq \vert \xi_{n+p+j}(\omega) \vert \geq \frac{\delta C_M}{2}e^{\gamma_M(j-1)}\frac{1}{r^{\kappa_2}}.
\]
Taking the logarithm, using (BA), and making sure that $m_0$ is large enough, we therefore have
\[
j \leq 1 + \frac{1}{\gamma_M}(\kappa_2 \log r + \Delta + \log 4 - \log C_M) \lesssim r \lesssim \log n < \frac{1}{2}(\log n)^2,
\]
as long as $j \leq L$ and $j \leq (\log n)^2$. This tells us that $j \leq L$ must break before $j \leq (\log n)^2$; in particular there is a constant $\kappa_3 > 0$ such that $L \leq \kappa_3 r$.
\end{proof}
\begin{Rem}\label{free_growth}
If the return $\xi_{n+p+L}(\omega)$ is inessential or essential, then there is no $\delta$-dependence in the growth factor; more generally, if the prerequisites of Lemma~\ref{bound_growth_lemma} are satisfied, then
\[
\vert \xi_{n+p+L}(\omega) \vert \geq \frac{C_M}{2}e^{\gamma_M(L-1)} \frac{1}{r^{\kappa_2}}\frac{\vert \xi_{n}(\omega)\vert}{\vert I \vert}.
\]
\end{Rem}

Before considering iterations of $\omega = \omega_k^{rl} \subset \Delta_k$ from $m = m_k^{rl}$ to $m_{k+1}^{rl}$, we make the following observation that as long as (BA) is assumed in a time window $[n,2n]$, the derivative will not drop too much.
\begin{Lemma}\label{window_lemma}
Suppose that $a$ is a parameter such that
\begin{equation}\label{window_derivative}
\vert \partial_x F^j(1;a) \vert \geq C e^{\gamma' j} \qquad (j = 0,1,\dots,n-1),
\end{equation}
with $\gamma' \geq \gamma/3$.
Then, if (BA) is satisfied up to time $2n$, we have
\[
\vert \partial_x F^{n+j}(1;a) \vert \geq Ce^{(\gamma'/3)(n+j)} \qquad (j = 0,1, \dots,n-1).
\]
In other words, if (BA) and (BE)(1) [(BE)(2)] are satisfied up to time $n$ then (BE)(2) [(BE)(3)] is satisfied up to time $2n$, as long as (BA) is.
\end{Lemma}

\begin{proof}
The proof is based on the fact that we trivially have no loss of derivative during the bound and free periods. Indeed suppose $\xi_{n'}(a) \sim e^{-r}$, with $n' \geq n$ and let $p$ be the bound period (here we use pointwise binding, see Remark~\ref{pointwise_binding}), and $L$ the free period. Moreover we assume that $n' + p + L < 2n$; in particular this implies $p < n$ and we can use (\ref{window_derivative}) during this period. Introducing $D_p = \vert \partial_x F^p(1;a) \vert$ and using similar calculations as in Lemma~\ref{bound_growth_lemma} (e.g. the equality in (\ref{bound_estimate}) and reversed inequality in (BC)) we find that
\[
e^{-2r} D_p \gtrsim a \eta^2 \vert \partial_x F^p(1-a\eta^2;a) \vert \geq \frac{\vert \xi_{p+1}(a) \vert}{10(p+1)^2} \gtrsim \frac{1}{(p+1)^{2 + \hat{e}}},
\]
where we used (BA) (or (PR)). Since $p < n$ we are free to use (\ref{window_derivative}) and therefore the above inequalities yield
\[
e^{-r} D_p \gtrsim D_p^{1/2} \frac{1}{\sqrt{(p+1)^{2 + \hat{e}}}} \gtrsim \frac{e^{(\gamma'/2)p}}{\sqrt{(p+1)^{2 + \hat{e}}}} \geq C_M^{-1},
\]
provided $\delta$ is small enough. Here in the last inequality we used the lower bound in (\ref{bound_period}) (see Remark \ref{lower_bound_bound_period_remark}). Assuming $\xi_{n'+p+L}(a)$ is a return (and that $n'+p+L < 2n$), we therefore have
\begin{align*}
\vert \partial_x F^{p+L}(\xi_{n'}(a);a) \vert &\geq 2a\vert \xi_{n'}(a) \vert \vert \partial_x F^p(1-a\xi_{n'}(a)^2;a) \vert \vert \partial_x F^{L-1}(\xi_{n'+p+1}(a);a)\vert \\
&\gtrsim e^{-r} D_p C_M e^{\gamma_M(L-1)} \\
&\geq 1.
\end{align*}
We conclude that the combination of a return, a bound period, and a free period does not decrease the derivative. 

Let us now follow a parameter $a$ satisfying (\ref{window_derivative}) and (BA) up to time $2n$. If the iterates $\xi_{n+j}(a)$ are always outside $(-\delta,\delta)$ then
\begin{align*}
\vert \partial_x F^{n+j}(1;a) \vert &= \vert \partial_x F^{n-1}(1;a) \vert \vert \partial_x F^{j+1}(\xi_n(a);a) \vert \\
&\geq Ce^{\gamma'(n-1)} \delta C_M e^{\gamma_M(j+1)} \\
&\geq Ce^{(\gamma'/3)(n+j)} \delta C_M e^{(2\gamma'/3)(n+j)} \\
&\geq Ce^{(\gamma'/3)(n+j)} \qquad (j = 0,1,\dots,n-1),
\end{align*}
provided $m_0$ is big enough.

Otherwise, the worst case is if we have a short free period followed by a return, a bound period, a free period, and so on, and which ends with a return together with a short bound period. In this case, using the above argument, the estimate is as follows:
\begin{align*}
\vert \partial_x F^{n+j}(1;a) \vert &\geq \vert \partial_x F^{n-1}(1;a) \vert \cdot C_M \cdot 1 \cdot 1 \cdots 1 \cdot 2a\vert \xi_{n+j}(a) \vert \cdot C \\
&\geq Ce^{\gamma'(n-1)} C_M C 2a \frac{\delta_{n+j}}{3} \\
&\geq Ce^{(\gamma'/3)(n+j)} C_M C \frac{2}{3a} e^{(\gamma'/3)n - \overline{e}\log(2n)} \\
&\geq Ce^{(\gamma'/3)(n+j)} \qquad (j = 0,1,\dots,n-1),
\end{align*}
provided $m_0$ is big enough. This proves the lemma.
\end{proof}

\subsection{From the \texorpdfstring{$k^\text{th}$}{k:th} complete return to the first inessential return}
If $\omega \subset T_k$ then we have already reached an escape situation and proceed accordingly as is described below in the section about escape. We therefore assume $\omega \subset N_k$ and $\xi_m(\omega) = I_{r_0l} \subset (-4\delta,4\delta)$.

If it happens that for some $j \leq p$
\[
\xi_{m+j}(\omega) \cap (-\delta_{m+j}/3,\delta_{m+j}/3) \neq \emptyset,
\]
then we stop and consider this return complete. If not, we notice that $\xi_{m+p}(\omega)$ can not be a return, unless it is escape or complete; indeed we would otherwise have $\vert \xi_{m+p+1}(\omega) \vert < \vert \xi_{m+p}(\omega)\vert$, due to the fact that we return close to the critical point, thus contradicting the definition of the bound period. We therefore assume that $\xi_{m+p}(\omega)$ does not intersect $(-\delta,\delta)$.

Up until the next return we will therefore experience an orbit outside of $(-\delta,\delta)$, i.e. we will be in a free period. After the free period, our return is either inessential, essential, escape, or complete. In the next section we consider the situation of an inessential return. 

\subsection{From the first inessential return to the first essential return}
Let $i_1 = m + p_0 + L_0$ denote the index of the first inessential return to $(-\delta,\delta)$. We will keep iterating $\xi_{i_1}(\omega)$ until we once again return. If this next return is again inessential, we denote its index by $i_2 = i_1 + p_1 + L_1$, where $p_1$ and $L_1$ are the associated bound period and free period, respectively. Continuing like this, let $i_j$ be the index of the $j^\text{th}$ inessential return.

The following lemma gives an upper bound for the total time spent doing inessential returns (compare with Lemma~2.3 in \cite{BC2}).
\begin{Lemma}[Inessential Length]\label{inessential_period_lemma}
Let $\xi_n(\omega) \subset I_{r+1} \cup I_r \cup I_{r-1}$ with $n$ being the index of a complete return or an essential return, and suppose that (BA) and (BE)(2) are satisfied up to time $n$. Then there exists a constant $\kappa_4 > 0$ such that the total time $o$ spent doing inessential returns satisfy
\[
o \leq \kappa_4 r.
\]
\end{Lemma}
\begin{proof}
Let $i_1 = n + p + L$ be the index of the first inessential return, i.e. $\xi_{i_1}(\omega) \subset I_{r_1}$, with $I_{r_1}$ being the host interval. From Lemma~\ref{bound_period_lemma} and Lemma~\ref{free_period_lemma}, together with (BA), we have that
\[
i_1 = n + p + L \leq n + (\kappa_1 + \kappa_3)r \leq 2n,
\]
provided $m_0$ is large enough. We can therefore apply Lemma~\ref{window_lemma} and conclude that (BE)(3) is satisfied at time $i_1$. To this first inessential return we associate a bound period of length $p_1$ (satisfying $p_1 \leq \kappa_1 r_1$ due to the fact that (BE)(3) is satisfied time $i_1$) and a free period of length $L_1$. We let $i_2 = i_1 + p_1 + L_1$ denote the index of the second inessential return. Continuing like this, we denote by $i_j = i_{j-1} + p_{j-1} + L_{j-1}$ the index of the $j^\text{th}$ inessential return. With $o_j$ denoting the total time spent doing inessential returns up to time $i_j$, we have that $o_j = i_{j} - i_1 = \sum_{k = 1}^{j-1} (p_k + L_k)$. Suppose that the return with index $i_s$ is the first that is not inessential. We estimate $o = o_s$ as follows. Suppose that $o_j$ is as above and that $p_k \leq \kappa_1 r_k$ for $k = 1,2,\dots,j-1$. Using Remark \ref{free_growth} we find that
\begin{equation}\label{inessential_growth_factor}
\frac{\vert \xi_{i_{k+1}}(\omega)\vert}{\vert \xi_{i_k}(\omega)\vert} \geq \frac{C_M e^{\gamma_M (L_k-1)}}{2 r_k^{\kappa_2}\vert I_{r_k}\vert} \geq \frac{C_M}{2} \frac{e^{\gamma_M (L_k-1) + r_k}}{r_k^{\kappa_2}},
\end{equation}
and therefore
\begin{equation}\label{inessential_product}
2 \geq \vert \xi_{i_j}(\omega) \vert = \vert \xi_{i_1}(\omega)\vert \prod_{k=1}^{j-1} \frac{\vert \xi_{i_{k+1}}(\omega)\vert}{\vert \xi_{i_k}(\omega) \vert} \geq \frac{\delta C_Me^{\gamma_M}}{2 r^{\kappa_2}} \prod_{k=1}^{j-1} \frac{C_M}{2} \frac{e^{\gamma_M (L_k-1) + r_k}}{r_k^{\kappa_2}}.
\end{equation}
Here the $\delta$ is added to make sure that the estimate also holds for the last free orbit, when the return can be escape or complete. This gives us a rather poor estimate, but since $p \lesssim r$ it is good enough.

Taking the logarithm of (\ref{inessential_product}) we find that
\[
\sum_{k=1}^{j-1} \left(\log C_M - \log 2 + \gamma_M (L_k-1) + r_k - \kappa_2 \log r_k \right) \leq \kappa_2 \log r + \Delta + \operatorname{const.}.
\]
Provided $\delta$ is small enough we have $r_k \geq 4\kappa_2 \log r_k$ and $r_k \geq -\log \delta > -2(\log C_M + \gamma_M + \log 2)$. Therefore, using $p_k \leq \kappa_1 r_k$, we find that
\begin{align*}
o_j = i_j - i_1 = \sum_{k=1}^{j-1}( p_k + L_k ) \leq \kappa_4 r,
\end{align*}
with $\kappa_4$ being an absolute constant. In particular
\[
i_j = i_1 + o_j \leq 2n,
\]
and therefore (BE)(3) is still valid at time $i_j$. Consequently the associated bound period satisfies $p_j \leq \kappa_1 r_j$, and the above argument can therefore be repeated. With this we conclude that $o_s \leq \kappa_4 r$.
\end{proof}

We proceed in the next section with describing the situation if our return is assumed to be essential.

\subsection{From the first essential return to the first escape return}
With $n_1$ denoting the index of the first essential return, we are in the following situation
\begin{align*}
&\xi_{n_1}(\omega) \cap I_{rl} \neq \emptyset, \quad \vert \xi_{n_1}(\omega)\vert \geq \vert I_{rl}\vert,\\
&\text{and}\ \xi_{n_1}(\omega) \subset (-4\delta,4\delta) \smallsetminus (-\delta_{n_1}/3,\delta_{n_1}/3),
\end{align*}
for some $r,l$. At this point, in order not to lose too much distortion, we will make a partition of as much as possible, and keep iterating what is left. That is, we will consider iterations of larger partition elements $I_r = (e^{-r-1},e^{-r}) \subset (-4\delta,4\delta)$, and we establish an upper bound for the \emph{number} of essential returns needed to reach an escape return or a complete return.

Let $\Omega_1 = \xi_{n_1}(\omega)$ and let $I_1 = I_{r_1} \subset \Omega_1$, for smallest such $r_1$. (In fact, we extend $I_1$ to the closest endpoint of $\Omega_1$, and therefore have $I_1 \subset I_{r_1} \cup I_{r_1-1}$.) If there is no such $r$, we instead let $I_1 = \Omega_1$. Moreover, let $\omega^1$ be the interval in parameter space for which $\xi_{n_1}(\omega^1) = I_{1}$. The interval $I_1$ is referred to as the essential interval, and this is the interval we will iterate. If $\hat{\omega} = \omega \smallsetminus \omega^1$ is nonempty we make a partition
\[
\hat{\omega} = \bigcup_{r,l} \omega^{rl} \subset \Delta_{k+1},
\]
where each $\omega^{rl}$ is such that $I_{rl} \subset \xi_{n_1}(\omega^{rl}) = I_{rl}\cup I_{r'l'} \subset (-4\delta,4\delta)$. (If there is not enough left for a partition, we extend $I_1$ further so that $I_1 \subset I_{r_1+1}\cup I_{r_1} \cup I_{r_1-1}$.) Notice that, since the intervals $I_r$ are dyadic, the proportion of what remains after partitioning satisfies
\begin{equation}\label{partitioned_free_return}
\frac{\vert I_1 \vert}{\vert \Omega_1 \vert} \geq 1 - \frac{1}{e} \geq \frac{1}{2}.
\end{equation}
We associate for each partitioned parameter interval $\omega^{rl}$ the complete return time $n_1$ (even though nothing is removed from these intervals). From the conclusions made in the previous sections we know that
\[
n_1 = m + p_0 + L_0 + o_0 \leq m + (\kappa_1 + \kappa_3 + \kappa_4)r_0 \leq 2m,
\] 
provided $m_0$ is large enough. In particular Lemma~\ref{window_lemma} tells us that (BE)(2) is satisfied up to time $n_1$ for all $a \in \omega$. At this step, to make sure that (BE)(1) is satisfied for our partitioned parameter intervals $\omega^{rl} \subset \Delta_{k+1}$, we make the following rule (compare with the initial iterates at the beginning of the induction step). If there is no $a' \in \omega$ such that
\[
\vert \partial_x F^{n_1-1}(1;a')\vert \geq C_B e^{\gamma_B(n_1-1)},
\]
then we remove the entire interval. If there is such a parameter, on the other hand, using Lemma~\ref{main_distortion_lemma} we have that
\begin{align*}
\vert \partial_x F^{n_1 - 1}(1;a) \vert &\geq D_1^{-(\log^* m)^2} \vert \partial_x F^{n_1-1}(1;a')\vert \\
&\geq C_B \exp\left\{ \left(\gamma_B - \frac{(\log^* m)^2}{n_1-1}\log D_1 \right)(n_1-1) \right\} \\
&\geq Ce^{\gamma(n_1-1)},
\end{align*}
provided $m_0$ is large enough.

With the above rules applied at each essential return to come, we now describe the iterations. Since $\xi_{m}(\omega) = I_{r_0l}$, using Lemma~\ref{bound_growth_lemma} we know that the length of $\Omega_1$ satisfies
\[
\vert \Omega_1 \vert \geq \frac{C_M e^{\gamma_M}}{2}\frac{1}{r_0^{\kappa_2}} \geq \frac{1}{r_0^{\kappa_2 + 1}}.
\]
Notice that since $e^{-r_1+1} \geq \vert \Omega_1 \vert$ we have that $r_1 \leq 2\kappa_2 \log r_0$. Iterating $I_{1}$ with the same rules as before, we will eventually reach a second noninessential return, and if this return is essential we denote its index by $n_2$. This index constitutes the addition of a bound period, a free period, and an inessential period: $n_2 = n_1 + p_1 + L_1 + o_1$. Similarly as before, we let $\Omega_2 = \xi_{n_2}(\omega^1)$, and let $I_{2} \subset \Omega_2$ denote the essential interval of $\Omega_2$. Let $\omega^2 \subset \omega^1$ be such that $\xi_{n_2}(\omega^2) = I_{2}$, and make a complete partition of $\omega^1 \smallsetminus \omega^2$. By applying Lemma~\ref{bound_growth_lemma} again, we find that
\[
\vert \Omega_2 \vert \geq \frac{1}{r_1^{\kappa_2 + 1}} \geq \frac{1}{(2\kappa_2 \log r_0)^{\kappa_2 + 1}}.
\]

If we have yet to reach an escape return or a complete return, let $n_j$ be the index of the $j^\text{th}$ essential return, and realise that we are in the following situation
\begin{equation}\label{jth outer return}
\xi_{n_j}(\omega^j) = I_j \subset \Omega_j = \xi_{n_j}(\omega^{j-1})\quad \text{and} \quad \vert \Omega_j \vert \geq \frac{1}{r_{j-1}^{\kappa_2 + 1}}.
\end{equation}
Introducing the function $r \mapsto 2\kappa_2 \log r = \varphi(r)$, we see from the above that $r_j \leq \varphi^j(r_0)$. The orbit $\varphi^j(r_0)$ will tend to the attracting fixed point $\hat{r} = -2\kappa_2 W(-1/(2\kappa_2))$, where $W$ is the Lambert $W$ function. The following simple lemma gives an upper bound for the number of essential returns needed to reach an escape return or a complete return.
\begin{Lemma}
Let $\varphi(r) = 2\kappa_2 \log r$, and let $s = s(r)$ be the integer defined by
\[
\log_{s} r \leq 2\kappa_2 \leq \log_{s-1} r.
\]
Then
\[
\varphi^s(r) \leq 12\kappa_2^2.
\]
\end{Lemma}
\begin{proof}
Using the fact that $3 \leq 2\kappa_2 \leq \log_j r$, for $j = 0,1,\dots, s-1$, it is straightforward to check that
\begin{equation}\label{varphi_upper_bound}
\varphi^j(r) \leq 6 \kappa_2 \log_j r.
\end{equation}
Therefore
\[
\varphi^s(r) \leq 2\kappa_2 \log\left( 6 \kappa_2 \log_{s-1}r \right) = 2\kappa_2 \left( \log 3 + \log 2\kappa_2 + \log_s r\right) \leq 12 \kappa_2^2.
\]
\end{proof}
Given $s = s(r_0)$ as in the above lemma we have that $r_s \leq \varphi^s(r_0) \leq 12\kappa_2^2$. By making sure $\delta$ is small enough we therefore conclude that
\[
\vert \Omega_{s+1} \vert \geq \frac{1}{(12\kappa_2^2)^{\kappa_2 + 1}} \geq 4\delta.
\]
To express $s$ in terms of $r_0$ we introduce the so-called iterated logarithm, which is defined recursively as
\[
\log^* x = \begin{cases} 0 &\mbox{if}\ x \leq 1, \\ 1 + \log^* \log x &\mbox{if}\ x > 1. \end{cases}
\]
That is, $\log^* x$ is the number of times one has to apply to logarithm to $x$ in order for the result to be less than or equal to one. 

Since $s$ satisfies $\log_s r_0 \leq 2\kappa_2 \leq \log_{s-1} r_0$ and since $2\kappa_2 > 1$, we have 
\begin{equation}\label{number_of_free_returns}
s \leq \log^* r_0 \leq \log^* m.
\end{equation}

We finish by giving an upper bound for the index of the first escape return (or $(k+1)^{\text{th}}$ complete return), i.e. we wish to estimate
\[
n_{s+1} = m + \sum_{j=0}^{s} \left(p_j + L_j + o_j\right).
\]
From Lemma~\ref{bound_period_lemma}, Lemma~\ref{free_period_lemma}, and Lemma~\ref{inessential_period_lemma}, we have that
\[
p_j \leq \kappa_1 r_j,\quad L_j \leq \kappa_3 r_j,\quad \text{and}\quad o_j \leq \kappa_4 r_j.
\]
Together with the inequalities $r_j \leq \varphi^j(r_0)$ and (\ref{varphi_upper_bound}),  we find that
\begin{align*}
\sum_{j=0}^{s}\left( p_j + L_j + o_j\right) &\lesssim r_0 + \sum_{j=1}^{s} \varphi^j(r_0) \\
&\lesssim r_0 + \sum_{j=1}^{s} \log_j r_0 \\
&\lesssim r_0.
\end{align*}
Using (BA) we conclude that $n_{s+1} - m \lesssim \log m$, provided $m_0$ is large enough.

\subsection{From the first escape return to the \texorpdfstring{$(k+1)^\text{th}$}{(k+1):th} complete return}
Keeping the notation from the previous section, $\Omega_{s+1} = \xi_{n_{s+1}}(\omega^s)$ is the first escape return, satisfying
\begin{align*}
&\Omega_{s+1} \cap (-\delta,\delta) \neq \emptyset, \quad \Omega_{s+1} \cap (-\delta_{n_s}/3,\delta_{n_s}/3) = \emptyset, \\
&\text{and} \quad \vert \Omega_{s+1} \smallsetminus (-\delta,\delta) \vert \geq 3\delta.
\end{align*}
We will keep iterating $\omega^s$ until we get a complete return, and we show below that this must happen within finite (uniform) time. In order to not run into problems with distortion we will, as in the case of essential returns, whenever possible make a partition of everything that is mapped inside of $(-\delta,\delta)$, and the corresponding parameter intervals will be a part of $\Delta_{k+1}$; i.e. at time $n_{s+1+j} = n_{s+1}+j$ let $I_{s+1+j} = \Omega_{s+1+j} \smallsetminus \left(\Omega_{s+1+j} \cap (-\delta,\delta)\right)$, let $\omega^{s+1+j}$ be such that $\xi_{n_{s+1+j}}(\omega^{s+1+j}) = I_{s+1+j}$, and make a partition of $\omega^{s+j}\smallsetminus \omega^{s+1+j}$. As in the case of essential returns, we associate to each partitioned parameter interval the complete time $n_{s+1+j}$, and as before we make sure that at these times (BE)(1) is satisfied.

Let $\omega_e = \omega_L \cup \omega_M \cup \omega_R$ be the disjoint union of parameter intervals for which
\[
\xi_{n_{s+1}}(\omega_L) = (\delta,2\delta), \quad \xi_{n_{s+1}}(\omega_M) = (2\delta,3\delta),\quad \text{and} \quad \xi_{n_{s+1}}(\omega_R) = (3\delta,4\delta).
\]
Clearly it is enough to show that $\omega_e$ reaches a complete return within finite time. Let $t_*$ be the smallest integer for which
\[
C_M e^{\gamma_M t_*} \geq 4.
\]
If $\delta$ is small enough, and if $\vert \omega_0 \vert = 2\epsilon$ is small enough, we can make sure that 
\[
\xi_{n_{s+1} + j}(\omega_e) \cap (-2\delta,2\delta) = \emptyset \qquad (1 \leq j \leq t_*). 
\]
Suppose that, for some $j \geq t_*$, $\xi_{n_{s+1+j}}(\omega_e) \cap (-\delta,\delta) \neq \emptyset$, and that this return is not complete. Assuming that $\omega_L$ returns we can not have $\xi_{n_{s+1+j}}(\omega_L) \subset (-2\delta,2\delta)$. Indeed, if this was the case, then (using Lemma~\ref{outside_expansion_lemma} and parameter independence)
\[
\vert \xi_{n_{s+1+j}}(\omega_L)\vert  > 2 \vert \xi_{n_{s+1}}(\omega_L)\vert > 2\delta,
\]
contradicting the return not being complete. We conclude that after partitioning what is mapped inside of $(-\delta,\delta)$, what is left is of size at least $\delta$, and we are back to the original setting. In particular, $\omega_M$ did not return to $(-\delta,\delta)$. Repeating this argument, $\omega_L$ and $\omega_R$ will return, but $\omega_M$ will stay outside of $(-\delta,\delta)$. (Here we abuse the notation: if $\omega_L$ returns we update it so that it maps onto $(\delta,2\delta)$, and similarly if $\omega_R$ returns.) Due to Lemma \ref{outside_expansion_lemma} we therefore have
\[
2 \geq \vert \xi_{n_{s+1+j}}(\omega_M) \vert \gtrsim \vert \xi_{n_{s+1}}(\omega_M) \vert \delta C_M e^{\gamma_M j} \geq \delta^2 C_M e^{\gamma_M j} \qquad (j \geq 0),
\]
and clearly we must reach a complete return after $j = t$ iterations, with
\[
t \lesssim \frac{2\Delta - \log C_M}{\gamma_M}.
\]
With this we conclude that if $m_0$ is large enough then there exists a constant $\kappa > 0$ such that
\begin{equation}\label{complete_time_estimate}
m_{k+1}^{rl} \leq m_k^{rl} + \kappa \log m_k^{rl}.
\end{equation}

We finish by estimating how much of $\Omega_{s+1+j}$ is being partitioned at each iteration. By definition of an escape return we have that $\vert \Omega_{s+1}\vert \geq 3\delta$, and since it takes a long time for $\omega_e$ to return, the following estimate is valid:
\begin{equation}\label{partitioned_escape_return}
\frac{\vert I_{s+1+j} \vert}{\vert \Omega_{s+1+j} \vert} \geq \frac{\vert \Omega_{s+1+j} \vert - \delta}{\vert \Omega_{s+1+j} \vert } \geq 1 -\frac{1}{3} = \frac{2}{3}.
\end{equation}

\subsection{Parameter exclusion}
We are finally in the position to estimate how much of $\omega$ is being removed at the next complete return. Up until the first free return, nothing is removed (unless we have a bound return, for which we either remove nothing, or remove enough to consider the return complete). Let $E$ be what is removed in parameter space, and write $\omega = \omega^0$. Taking into account what we partition in between $m_k$ and $m_{k+1}$ we have that
\[
\frac{\vert E \vert}{\vert \omega^0 \vert} = \frac{\vert E \vert}{\vert \omega^{s+t} \vert} \prod_{\nu = 0}^{t-1} \frac{\vert \omega^{s+1+\nu}\vert}{\vert \omega^{s+\nu}\vert} \prod_{\nu=0}^{s-1} \frac{\vert \omega^{1 + \nu}\vert}{\vert \omega^\nu \vert}
\]
Using the the mean value theorem we find that for each factor in the above expression
\begin{align*}
\frac{\vert \omega^j\vert}{\vert \omega^{j-1}\vert} &= \frac{\vert a_j - b_j \vert}{\vert a_{j-1} - b_{j-1} \vert} \\
&= \frac{\vert a_j - b_j \vert}{\vert \xi_{n_j}(a_j) - \xi_{n_j}(b_j)\vert} \frac{\vert \xi_{n_j}(a_{j-1}) - \xi_{n_j}(b_{j-1})\vert}{\vert a_{j-1} - b_{j-1}\vert} \frac{\vert \xi_{n_j}(a_j) - \xi_{n_j}(b_j)\vert}{\vert \xi_{n_j}(a_{j-1}) - \xi_{n_j}(b_{j-1})\vert} \\
&= \frac{1}{\vert \partial_a \xi_{n_j}(c_j)\vert} \vert \partial_a \xi_{n_j}(c_{j-1})\vert \frac{\vert I_{j}\vert}{\vert \Omega_j \vert} \\
&= \frac{\vert \partial_x F^{n_j-1}(1;c_j)\vert}{\vert \partial_a \xi_{n_j}(c_j)\vert} \frac{\vert \partial_a \xi_{n_j}(c_{j-1})\vert}{\vert \partial_x F^{n_j-1}(1;c_{j-1})\vert} \frac{\vert \partial_x F^{n_j-1}(1;c_{j-1})\vert}{\vert \partial_x F^{n_j-1}(1;c_j)\vert} \frac{\vert I_{j}\vert}{\vert \Omega_j \vert}.
\end{align*}
Making use of Lemma~\ref{tsujii_distortion} and Lemma~\ref{main_distortion_lemma}, we find that
\[
\frac{\vert \omega^j \vert}{\vert \omega^{j-1}\vert} : \frac{\vert I_{j} \vert}{\vert \Omega_j \vert} \sim D_A D_1^{(\log^* m_k)^2},
\]
and therefore, using (\ref{number_of_free_returns}), (\ref{partitioned_free_return}), and (\ref{partitioned_escape_return}), there is, provided $m_0$ is large enough, an absolute constant $0 < \tau < 1$ such that
\[
\frac{\vert E \vert}{\vert \omega^0 \vert} \geq \frac{(\delta_{m_{k+1}}/3)}{1}\left(\frac{1}{3}D_A^{-1} D_1^{-(\log^* m_k)^2}\right)^{t + \log^* m_k} \geq \delta_{m_{k+1}}\tau^{(\log^* m_{k+1})^3}.
\]
In particular, for the remaining interval $\hat{\omega} = \omega \smallsetminus E$ we have that
\begin{equation}\label{what_is_left}
\vert \hat{\omega} \vert \leq \vert \omega \vert (1- \delta_{m_{k+1}} \tau^{(\log^* m_{k+1})^3}).
\end{equation}

\section{Main Distortion Lemma}

Before giving a proof of Theorem~\ref{Theorem_B}, we give a proof of the much important distortion lemma that, together with Lemma~\ref{tsujii_distortion}, allow us to restore derivative and to estimate what is removed in parameter space at the $(k+1)^\text{th}$ complete return. The proof is similar to that of Lemma~5 in \cite{BC1}, with the main difference being how we proceed at essential returns. As will be seen, our estimate is unbounded.

If not otherwise stated, the notation is consistent with that of the induction step. Recall that
\[
\Delta_k = N_k \cup T_k,
\]
with $\omega_k \subset N_k$ being mapped onto some $I_{rl} \subset (-4\delta,4\delta)$, and $\omega_k \subset T_k$ being mapped onto an interval $\pm (\delta,x)$ with $\vert x-\delta \vert \geq 3\delta$. Moreover, we let $m_{k+1}(a,b)$ denote the largest time for which parameters $a,b \in \omega_k$ belong to the same parameter interval $\omega_k^j \subset \omega_k$, e.g. if $a,b \in \omega_k^j$ then $m_{k+1}(a,b) \geq n_{j+1}$.

\begin{Lemma}[Main Distortion Lemma]\label{main_distortion_lemma}
Let $\omega_k \subset \Delta_k$, and let $m_k$ be the index of the $k^\text{th}$ complete return. There exists a constant $D_1 > 1$ such that, for $a,b \in \omega_k$ and $j < m_{k+1} = m_{k+1}(a,b)$,
\[
\frac{\vert \partial_x F^j(1;a)\vert}{\vert \partial_x F^j(1;b)\vert} \leq D_1^{(\log^* m_k)^2}.
\]
\end{Lemma}
\begin{proof}
Using the chain rule and the elementary inequality $x+1 \leq e^x$ we have
\begin{align*}
\frac{\vert \partial_x F^j(1;a)\vert}{\vert \partial_x F^j(1;b)\vert} &= \prod_{\nu= 0}^{j-1} \frac{\vert \partial_x F(F^\nu(1;a);a) \vert}{\vert \partial_x F(F^\nu(1;b);b) \vert} \\
&= \left(\frac{a}{b}\right)^j \prod_{\nu =1}^j \frac{\vert \xi_\nu(a) \vert}{\vert \xi_\nu(b)\vert} \\
&\leq \left( \frac{a}{b} \right)^j \prod_{\nu=1}^j \left(\frac{\vert \xi_\nu(a) - \xi_\nu(b) \vert}{\vert \xi_\nu(b)\vert} + 1\right) \\
&\leq \left(\frac{a}{b}\right)^j \exp\left(\sum_{\nu=1}^j \frac{\vert \xi_\nu(a) - \xi_\nu(b) \vert}{\vert \xi_\nu(b) \vert}\right).
\end{align*}
We claim that the first factor in the above expression can be made arbitrarily close to $1$. To see this, notice that
\[
\left(\frac{a}{b}\right)^j \leq \left(1 +  \vert \omega_k \vert \right)^j.
\]
Using (BE)(1) and Lemma~\ref{tsujii_distortion} we have that $\vert \omega_k\vert \lesssim e^{-\gamma m_k}$, and for $m_0$ large enough we have from (\ref{complete_time_estimate}) that $j < m_{k+1} \leq m_k + \kappa \log m_k \leq 2 m_k$; therefore
\[
\left(1+ \vert \omega_k \vert\right)^j \leq (1 + e^{-(\gamma/2)m_k})^{2m_k}.
\]
Since
\[
(1 + e^{-(\gamma/2)m_k})^{2m_k} \leq (1+e^{-(\gamma/2)m_0})^{2m_0} \to 1 \quad \text{as}\quad m_0 \to \infty,
\]
making $m_0$ larger if needed proves the claim. It is therefore enough to only consider the sum
\[
\Sigma = \sum_{\nu = 1}^{m_{k+1} -1} \frac{\vert \xi_\nu(a) - \xi_\nu(b)\vert}{\vert \xi_\nu(b)\vert}.
\]
With $m_k^* \leq m_{k+1}$ being the last index of a return, i.e. $\xi_{m_{k^*}}(\omega_k) \subset I_{r_k^*} \subset (-4\delta,4\delta)$, we divide $\Sigma$ as
\[
\Sigma = \sum_{\nu = 1}^{m_k^*-1} + \sum_{\nu = m_k^*}^{m_{k+1}-1} = \Sigma_1 + \Sigma_2,
\]
and begin with estimating $\Sigma_1$. 

The history of $\omega_k$ will be that of $\omega_0, \omega_1,\dots,\omega_{k-1}$. Let $\{t_j\}_{j=0}^N$ be all the inessential, essential, escape, and complete returns. We further divide $\Sigma_1$ as
\[
\sum_{\nu = 1}^{m_k^*-1} \frac{\vert \xi_\nu(a) - \xi_\nu(b)\vert}{\vert \xi_\nu(b)\vert} = \sum_{j=0}^{N-1}\sum_{\nu = t_j}^{t_{j+1}-1}\frac{\vert \xi_\nu(a) - \xi_\nu(b) \vert}{\vert \xi_\nu(b)\vert} = \sum_{j=0}^{N-1} S_j.
\]
The contribution to $S_j$ from the bound period is 
\[
\sum_{\nu=0}^{p_j} \frac{\vert \xi_{t_j+\nu}(a) - \xi_{t_j + \nu}(b)\vert}{\vert \xi_{t_j+\nu}(b)\vert} \lesssim \frac{\vert \xi_{t_j}(\omega)\vert}{\vert \xi_{t_j}(b) \vert} + \frac{\vert \xi_{t_j}(\omega) \vert}{\vert \xi_{t_j}(b)\vert}\sum_{\nu=1}^{p_j}\frac{e^{-2r_j} \vert \partial_xF^{\nu-1}(1;a)\vert}{\vert \xi_\nu(b) \vert}.
\]
Let $\iota = (\kappa_1 \log 4)^{-1}$ and further divide the sum in the above right hand side as
\[
\sum_{\nu=1}^{\iota p_j} + \sum_{\nu = \iota p_j + 1}^{p_j}.
\]
To estimate the first sum we use the inequalities $\vert \partial_x F^\nu \vert \leq 4^\nu$ and $\vert \xi_\nu(b) \vert \geq \delta_\nu/3 \gtrsim \nu^{-\overline{e}}$, and that $p_j \leq \kappa_1 r_j$, to find that
\begin{align*}
\sum_{\nu=1}^{\iota p_j} \frac{e^{-2r_j} \vert\partial_xF^{\nu-1}(1;a)\vert}{\vert \xi_\nu(b)\vert} &\lesssim e^{-2r_j}\sum_{\nu=1}^{\iota p_j} 4^\nu \nu^{\overline{e}} \\
&\lesssim e^{-2 r_j} 4^{\iota p_j}p_j^{\overline{e}} \\
&\lesssim \frac{r_j^{\overline{e}}}{e^{r_j}}.
\end{align*}
To estimate the second sum we use (BC) and the equality (\ref{bound_estimate}), and find that
\[
\sum_{\nu = \iota p_j + 1}^{p_j} \frac{e^{-2r_j} \vert \partial_x F^{\nu-1}(1;a)\vert}{\vert \xi_\nu(b)\vert} \lesssim \frac{1}{r_j^2}.
\]
Therefore the contribution from the bound period adds up to
\begin{align*}
\sum_{\nu = t_j}^{t_j+p_j} \frac{\vert \xi_\nu(a) - \xi_\nu(b)\vert}{\vert \xi_\nu(b) \vert} &\lesssim \frac{\vert \xi_{t_j}(\omega) \vert}{\vert \xi_{t_j}(b) \vert} + \frac{\vert \xi_{t_j}(\omega)\vert}{\vert \xi_{t_j}(b)\vert}\left(\frac{1}{r_j^2} + \frac{r_j^{\overline{e}}}{e^{r_j}}\right) \\
&\lesssim \frac{\vert \xi_{t_j}(\omega) \vert}{\vert \xi_{t_j}(b)\vert}.
\end{align*}

After the bound period and up to time $t_{j+1}$ we have a free period of length $L_j$ during which we have exponential increase of derivative. We wish to estimate
\[
\sum_{\nu = t_j + p_j + 1}^{t_{j+1}-1} \frac{\vert \xi_\nu(a) - \xi_\nu(b)\vert}{\vert \xi_\nu(b)\vert} = \sum_{\nu = 1}^{L_j-1} \frac{\vert \xi_{t_j+p_j + \nu}(a) - \xi_{t_j+p_j + \nu}(b) \vert}{\vert \xi_{t_j+p_j + \nu}(b)\vert}.
\]
Using the mean value theorem, parameter independence and Lemma~\ref{outside_expansion_lemma}, we have that for $1 \leq \nu \leq L_j-1$
\begin{align*}
\vert \xi_{t_{j+1}}(a) - \xi_{t_{j+1}}(b)\vert &= \vert \xi_{t_j+p_j+L_j}(a) - \xi_{t_j + p_j + L_j}(b)\vert \\
&\simeq \vert F^{L_j-\nu}(\xi_{t_j+p_j+\nu}(a);a) - F^{L_j -\nu}(\xi_{t_j+p_j+\nu}(b);a)\vert \\
&= \vert \partial_x F^{L_j-\nu}(\xi_{t_j+p_j+\nu}(a');a)\vert \vert \xi_{t_j+p_j+\nu}(a) - \xi_{t_j+p_j+\nu}(b)\vert \\
&\gtrsim e^{\gamma_M (L_j-\nu)}\vert \xi_{t_j+p_j+\nu}(a) - \xi_{t_j+p_j+\nu}(b)\vert,
\end{align*}
and therefore
\begin{equation}\label{free_period_contribution}
\vert \xi_{t_j+p_j+\nu}(a) - \xi_{t_j+p_j+\nu}(b) \vert \lesssim \frac{\vert \xi_{t_{j+1}}(a) - \xi_{t_{j+1}}(b) \vert}{e^{\gamma_M (L_j-\nu)}},
\end{equation}
provided $\xi_{t_{j+1}}(\omega)$ does not belong to an escape interval. If $\xi_{t_{j+1}}(\omega)$ belongs to an escape interval, then we simply extend the above estimate to $t_{j+2}, t_{j+3},\dots$, until we end up inside some $I_{rl} \subset (-4\delta,4\delta)$ (which will eventually happen, per definition of $m_k^*$). Hence we may disregard escape returns, and see them as an extended free period.

Since $\vert \xi_{t_{j+1}}(b) \vert \leq \vert \xi_{t_j+p_j+\nu}(b) \vert$ for $1 \leq \nu \leq L_j-1$, it follows from the above inequality that
\begin{align*}
\sum_{\nu=1}^{L_j-1} \frac{\vert \xi_{t_j+p_j+\nu}(a) - \xi_{t_j+p_j+\nu}(b)\vert}{\vert \xi_{t_j+p_j+\nu}(b)\vert} &\leq \frac{\vert \xi_{t_{j+1}}(a) - \xi_{t_{j+1}}(b)\vert}{\vert \xi_{t_{j+1}}(b)\vert}\sum_{\nu=1}^{L_j -1} e^{-\gamma_M (L_j-\nu)} \\
&\lesssim \frac{\vert \xi_{t_{j+1}}(a) - \xi_{t_{j+1}}(b)\vert}{\vert \xi_{t_{j+1}}(b) \vert},
\end{align*}
thus the contribution from the free period is absorbed in $S_{j+1}$.

What is left is to give an estimate of
\[
\sum_{\nu = m_0}^{m_k^*-1} \frac{\vert \xi_\nu(a) - \xi_\nu(b)\vert}{\vert \xi_\nu(b)\vert} \lesssim \sum_{j=0}^N \frac{\vert \xi_{t_j}(\omega)\vert}{\vert \xi_{t_j}(b)\vert} \lesssim \sum_{j=0}^N \frac{\vert \xi_{t_j}(\omega)\vert}{\vert I_{r_j}\vert},
\]
where, with the above argument, $\{t_j\}_{j=0}^N$ are now considered to be indices of inessential, essential, and complete returns only. Because of the rapid growth rate, we will see that among the returns to the same interval, only the last return will be significant. From Lemma~\ref{bound_growth_lemma} we have that $\vert \xi_{t_j+1}(\omega) \vert \gtrsim (e^{r_j}/r_j^{\kappa_2}) \vert \xi_{t_j}(\omega) \vert \gg 2 \vert \xi_{t_j}(\omega)\vert $, hence with $J(\nu)$ the last $j$ for which $r_j = \nu$,
\[
\sum_{j=0}^N \frac{\vert \xi_{t_j}(\omega)\vert}{\vert I_{r_j}\vert} = \sum_{\nu \in \{r_j\}} \frac{1}{\vert I_\nu\vert} \sum_{r_j = \nu} \vert \xi_{t_j}(\omega)\vert \lesssim \sum_{\nu \in \{r_j\}} \frac{\vert \xi_{t_{J(\nu)}}(\omega) \vert}{\vert I_{\nu}\vert}.
\]
If $t_{J(\nu)}$ is the index of an inessential return, then $\vert \xi_{t_{J(\nu)}}(\omega_k) \vert / \vert I_\nu \vert \lesssim \nu^{-2}$, and therefore the contribution from the inessential returns to the above left most sum is bound by some small constant. It is therefore enough to only consider the contribution from essential returns and complete returns. To estimate this contribution we may assume that $m_k^* \geq m_k$, and that $\xi_{m_k}(\omega) = I_{r_k l}$. Moreover, we assume that $\xi_{m_j}(\omega) = I_{r_j l}$ for all $j$.

With $n_{j,0} = m_j$ being the index of the $j^\text{th}$ complete return, and $n_{j,\nu} \in (m_j,m_{j+1})$ being the index of the $\nu^\text{th}$ essential return for which $\xi_{n_{j,\nu}}(\omega) \subset I_{r_{j,\nu}}$, we write
\[
\sum_{\nu \in \{r_j\}} \frac{\vert \xi_{t_{J(\nu)}}(\omega)\vert}{\vert I_\nu\vert} \lesssim \sum_{j = 0}^{k} \sum_{\nu = 0}^{\nu_j} \frac{\vert \xi_{n_{j,\nu}}(\omega) \vert}{\vert I_{r_{j,\nu}} \vert} = \sum_{j=0}^{k} S_{m_j}.
\]
For the last partial sum we use the trivial estimate $S_{m_k} \leq \log^* m_k$. To estimate $S_{m_j}$, for $j \neq k$, we realise that between any two free returns $n_{j,\nu}$ and $n_{j,\nu+1}$ the distortion is uniformly bound by some constant $C_1 > 1$. Therefore
\[
\frac{\vert \xi_{n_{k-1,\nu_{k-1}-j}}(\omega)\vert}{\vert I_{r_{k-1,\nu_{k-1}-j}} \vert} \leq \frac{C_1^j}{r_k^2},
\]
and consequently, since $\nu_j \leq \log^* r_j$ (see (\ref{number_of_free_returns})),
\[
S_{m_{k-1}} \leq \frac{C_2^{\log^* r_{k-1}}}{r_k^2},
\]
for some uniform constant $C_2 > 1$. Continuing like this, we find that
\begin{align*}
S_{m_{k-j}} &\leq \frac{C_2^{\log^* r_{k-j}}C_2^{\log^* r_{k-j+1}} \cdots C_2^{\log^* r_{k-1}}}{r_{k-j+1}^2 r_{k-j+2}^2 \cdots r_k^2} \\
&\leq \frac{C_2^{\log^*{r_{k-j}}}}{r_{k-j+1}^{3/2} r_{k-j+2}^{3/2} \cdots r_{k}^2},
\end{align*}
where we in the last inequality used the (very crude) estimate
\[
C_2^{\log^* x} \leq \sqrt{x}.
\]
Let us call the estimate of $S_{m_{k-j}}$ \emph{good} if $C_2^{\log^* r_{k-j}} \leq r_{k-j+1}$. For such $S_{m_{k-j}}$ we clearly have
\[
S_{m_{k-j}} \leq \frac{1}{\Delta^{j/2}}.
\]
Let $j_1 \geq 1$ be the smallest integer for which $S_{m_{k-j_1}}$ is not good, i.e. 
\[
\log^* r_{k-j_1} \geq (\log C_2)^{-1} \log r_{k-j_1+1} \geq (\log C_2)^{-1} \log\Delta.
\]
We call this the first \emph{bad} estimate, and for the contribution from $S_{m_{k-j_1}}$ to the distortion we instead use the trivial estimate 
\[
S_{m_{k-j_1}} \leq \log^* r_{k-j_1} \leq \log^* m_k.
\]
Suppose that $j_2 > j_1$ is the next integer for which
\[
C_2^{\log^* r_{k-j_2}} \geq r_{k-j_2+1}.
\]
If it turns out that
\[
C_2^{\log^* r_{k-j_2}} \leq r_{k-j_1},
\]
then 
\[
S_{m_k-j_2} \leq \frac{1}{\Delta^{j/2}},
\]
and we still call this estimate good. If not, then
\[
\log^* r_{k-j_2} \geq (\log C_2)^{-1} \log r_{k-j_1},
\]
and $j_2$ is the index of the second bad estimate.
Continuing like this, we get a number $s$ of bad estimates and an associated sequence $R_i = r_{k-j_i}$ satisfying
\begin{align*}
\log^* R_1 &\geq (\log C_2)^{-1} \log \Delta, \\
\log^* R_2 &\geq (\log C_2)^{-1} \log R_1, \\
&\vdots \\
\log^* R_s &\geq (\log C_2)^{-1} \log R_{s-1}.
\end{align*}
This sequence grows incredibly fast, and its not difficult to convince oneself that
\[
R_s \gg \underbrace{e^{e^{.^{.^{.^{e}}}}}}_{s\ \text{copies of}\ e}.
\]
In particular, since $R_s \leq m_k$ we find that
\[
s \ll \log^*R_s \leq \log^* m_k.
\]
We conclude that
\begin{align*}
\left(\sum_{\text{good}} + \sum_{\text{bad}}\right) S_{m_j} &\leq \sum_{j=1}^\infty \frac{1}{\Delta^{j/2}} + s \log^* m_k \\
&\lesssim (\log^* m_k)^2,
\end{align*}
hence
\[
\sum_{\nu = 1}^{m_k^*-1} \frac{\vert \xi_\nu(a) - \xi_\nu(b)\vert}{\vert \xi_\nu(b)\vert} \lesssim (\log^* m_k)^2.
\]

From $m_{k^*}$ to $m_{k+1} - 1$, the assumption is that we only experience an orbit outside $(-\delta,\delta)$. By a similar estimate as (\ref{free_period_contribution}) we find that for $\nu \geq 1$
\[
\vert \xi_{m_{k^*}+\nu}(a) - \xi_{m_{k^*}+\nu}(b)\vert \lesssim \frac{\vert \xi_{m_k-1}(a) - \xi_{m_k-1}(b)\vert}{\delta e^{\gamma_M(m_k - 1 - m_{k^*}-\nu)}},
\]
and therefore
\[
\sum_{\nu = m_k^*}^{m_{k+1}-1} \frac{\vert \xi_\nu(a) - \xi_\nu(b) \vert}{\vert \xi_\nu(b) \vert} \lesssim  1 + \frac{1}{\delta^2} \leq (\log^* m_k)^2,
\]
provided $m_0$ is large enough. This proves the lemma.
\end{proof}

\section{Proof of Theorem B}
Returning to the more cumbersome notation used in the beginning of the induction step, let $\omega_k^{rl} \subset \Delta_k$. We claim that a similar inequality as (\ref{what_is_left}) is still true if we replace $\hat{\omega}$ and $\omega$ with $\Delta_{k+1}$ and $\Delta_k$, respectively. To realise this write $\Delta_{k+1}$ as the disjoint union
\[
\Delta_{k+1} = \bigcup \omega_{k+1}^{rl} = \bigcup \hat{\omega}_k^{rl}.
\]
With $m_0$ being the start time, consider the sequence of integers defined by the equality
\[
m_{k+1} = \lceil m_k + \kappa \log m_k \rceil \qquad (k \geq 0),
\]
where $\lceil x \rceil$ denotes the smallest integer satisfying $x \leq \lceil x \rceil$. By induction, using (\ref{complete_time_estimate}),
\[
m_{k+1}^{rl} \leq m_k^{r'l'} + \kappa \log m_k^{r'l'} \leq m_k + \kappa \log m_k \leq m_{k+1}.
\]
Hence the sequence $(m_k)$ dominates every other sequence $(m_k^{rl})$, and therefore it follows from (\ref{what_is_left}) that
\begin{align*}
\vert \Delta_{k+1} \vert &= \sum \vert \hat{\omega}_k^{rl} \vert \\
&= \sum \vert \omega_k^{rl} \vert (1-\delta_{m_{k+1}^{rl}} \tau^{(\log^*m_{k+1}^{rl})^3}) \\
&\leq \left(\sum \vert \omega_k^{rl} \vert \right) (1-\delta_{m_{k+1}}\tau^{(\log^* m_{k+1})^3}) \\
&= \vert \Delta_k \vert (1-\delta_{m_{k+1}}\tau^{(\log^* m_{k+1})^3}).
\end{align*}

By construction
\[
\Lambda_{m_0}(\delta_n) \cap \CE(\gamma_B,C_B) \cap \omega_0 \subset \Delta_\infty = \bigcap_{k=0}^\infty \Delta_k,
\]
and therefore, to prove Theorem~\ref{Theorem_B}, it is sufficient to show that
\[
\prod_{k=0}^\infty(1-\delta_{m_k}\tau^{(\log^* m_{k})^3}) = 0.
\]
By standard theory of infinite products, this is the case if and only if
\[
\sum_{k = 0}^\infty \delta_{m_k}\tau^{(\log^* m_k)^3} = \infty.
\]
To evaluate the above sum we make use of the following classical result, due to Schl\"omilch~\cite{Schlomilch} (see also \cite{Knopp}, for instance).
\begin{Prop}[Schl\"omilch Condensation Test]
Let $q_0 < q_1 < q_2 \cdots$ be a strictly increasing sequence of positive integers such that there exists a positive real number $\alpha$ such that
\[
\frac{q_{k+1} - q_k}{q_k - q_{k-1}} < \alpha \qquad (k \geq 0).
\]
Then, for a nonincreasing sequence $a_n$ of positive nonnegative real numbers,
\[
\sum_{n = 0}^\infty a_n = \infty \quad \text{if and only if}\quad \sum_{k=0}^\infty (q_{k+1} -q_k)a_{q_k} = \infty.
\]
\end{Prop}
\begin{proof}
We have
\[
(q_{k+1}-q_k) a_{q_{k+1}} \leq \sum_{n=0}^{q_{k+1} - q_k -1} a_{q_k + n} \leq (q_{k+1}-q_k) a_{q_k},
\]
and therefore
\[
\alpha^{-1}\sum_{k=0}^\infty (q_{k+2}-q_{k+1}) a_{q_{k+1}} \leq \sum_{n = q_0}^\infty a_n \leq \sum_{k=0}^\infty (q_{k+1} - q_k) a_{q_k}.
\]
\end{proof}
Since $m_{k+1} - m_k \sim \log m_k$ is only dependent on $m_k$, we can easily apply the above result in a backwards manner. Indeed we have that
\begin{align*}
\frac{m_{k+1} - m_k}{m_{k} - m_{k-1}} &\leq \frac{\kappa \log m_k + 1}{\kappa \log m_{k-1}} \\
&\leq \frac{\kappa \log \left(m_{k-1} + \kappa\log m_{k-1} + 1\right) + 1}{\kappa \log m_{k-1}} \\
&\leq 1 + \frac{\operatorname{const.}}{\log m_0} \qquad (k \geq 0),
\end{align*}
and therefore with $q_k = m_k$ and $a_n = \delta_n \tau^{(\log^* n)^3}/\log n$, the preconditions of Schl\"omilch's test are satisfied. We conclude that
\[
\sum_{n = m_0}^\infty \frac{\delta_n}{\log n}\tau^{(\log^*n)^3} = \infty
\]
if and only if
\[
\sum_{k=0}^\infty (m_{k+1} - m_k)\frac{\delta_{m_k}}{\log m_k} \tau^{(\log^* m_k)^3} \sim \sum_{k=0}^\infty \delta_{m_k} \tau^{(\log^* m_k)^3} = \infty.
\]
This proves Theorem~\ref{Theorem_B}.

\bibliographystyle{alpha}
\bibliography{references}

\end{document}